\documentclass[12pt]{amsart}
\usepackage[margin=1in]{geometry}
\usepackage{amsmath,amsfonts,amsthm,amssymb,bbm}
\usepackage{graphicx,color,dsfont}
\usepackage{enumitem}
\usepackage{fourier}

\newtheorem{theorem}{Theorem}
\newtheorem{lemma}{Lemma}
\newtheorem{proposition}{Proposition}

\newtheorem{corollary}{Corollary}

\newtheorem{claim}{Claim}

 \newtheorem{thm}{Theorem}[section]
 
 \newtheorem{lem}[thm]{Lemma}
 
 \theoremstyle{definition}
 
 \theoremstyle{remark}
 \newtheorem{rem}[thm]{Remark}
 
 \numberwithin{equation}{section}

\newcommand{\vertiii}[1]{{\left\vert\kern-0.25ex\left\vert\kern-0.25ex\left\vert #1
    \right\vert\kern-0.25ex\right\vert\kern-0.25ex\right\vert}}

\newcommand{\f}[2]{\frac{#1}{#2}}

\newcommand{\al}{\alpha}
\newcommand{\be}{\beta}

\newcommand{\ga}{\gamma}
\newcommand{\Ga}{\Gamma}
\newcommand{\de}{\delta}

\newcommand{\De}{\Delta}

\newcommand{\eps}{\epsilon}

\newcommand{\p}{\partial}

\newcommand{\beq}{\begin{equation}}
\newcommand{\eeq}{\end{equation}}
\newcommand{\beqna}{\begin{eqnarray*}}
\newcommand{\eeqna}{\end{eqnarray*}}
\newcommand{\beqn}{\begin{equation*}}
\newcommand{\eeqn}{\end{equation*}}
\newcommand{\bp}{\begin{proof}}
\newcommand{\bprop}{\begin{proposition}}
\newcommand{\eprop}{\end{proposition}}
\newcommand{\bt}{\begin{theorem}}
\newcommand{\et}{\end{theorem}}
\newcommand{\bex}{\begin{Example}}
\newcommand{\eex}{\end{Example}}
\newcommand{\bc}{\begin{corollary}}
\newcommand{\ec}{\end{corollary}}
\newcommand{\bcl}{\begin{claim}}
\newcommand{\ecl}{\end{claim}}
\newcommand{\bl}{\begin{lemma}}
\newcommand{\el}{\end{lemma}}

\newcommand{\ep}{\epsilon}

\begin{document}

\title[Well-posedness and asymptotics  of a coordinate-free model]
{Well-posedness and asymptotics  of a coordinate-free model of flame fronts}

\thanks{The first author is grateful to the National Science Foundation for support under grant
DMS-1907684.}

\author{David M. Ambrose \and Fazel Hadadifard \and J. Douglas Wright}
 \address{Drexel University,
 	Department of Mathematics, Philadelphia, PA, 19104 USA}
\email{dma68@drexel.edu}
 \address{Drexel University,
 	Department of Mathematics, Philadelphia, PA,  19104 USA}
 \email{fh352@drexel.edu}
 \address{Drexel University,
 	Department of Mathematics, Philadelphia, PA,  19104 USA}
\email{jdw66@drexel.edu}



\date{\today}

\begin{abstract}
	We investigate a coordinate-free model of flame fronts introduced by Frankel and Sivashinsky;
	this model has a parameter $\alpha$ which relates to how unstable the front might be.
	We first prove short-time well-posedness of the coordinate-free model, for any value of $\alpha>0.$
	We then argue that near the threshold $\al \approx 1,$ the solution stays arbitrarily close to the solution of the weakly nonlinear Kuramoto--Sivashinsky (KS) equation, as long as the initial values are close.

\end{abstract}

\maketitle
\small

\section{Introduction}

The Kuramoto-Sivashinsky equation, 
\begin{equation}\label{ff}
f_t+ \f{1}{2} f_x^2+ (\al- 1) f_{x x}+ 4 f_{xxxx}= 0,
\end{equation}
is a weakly nonlinear model for flame fronts \cite{kuramoto}, \cite{sivashinsky}.  
Frankel and Sivashinsky have shown that it
can be formally derived from coordinate-free models \cite{frankelSivashinsky1987} of flame propagation.
In such a coordinate-free model, the normal velocity of the front is specified in
terms of intrinsic geometric information such as curvature and arclength.
One such model put forward by Frankel and Sivashinsky is
%
 \begin{equation}\label{vn}
V_n= 1+ (\al- 1) \kappa+ \bigg(1+ \f{1}{2} \al^2\bigg) \kappa^2+ \bigg(2 \al+ 5 \al^2- \f{1}{3} \al^3\bigg) \kappa^3+ \al^2 (\al+ 3)  \kappa_{ss},
\end{equation}
where $V_{n}$ is the normal velocity of the front, $\kappa$ is the curvature of the front, $s$ is arclength,
and $\alpha$ is a parameter measuring instability of the interface.
Frankel and Sivashinsky perform asymptotic analysis of \eqref{vn}  in the case $\al \approx 1,$ 
finding the simplified coordinate-free model
 \begin{equation}\label{vn1}
V_n= 1+ (\al- 1) \kappa+  4  \kappa_{ss}.
\end{equation}

As discussed by Brauner et al. \cite{braunerEtAl},
there are two primary destabilization mechanisms for premixed gas combustion: hydrodynamic instability (stemming
from thermal expansion of the gas), and thermal-diffusive instability. 
The derivation of the models \eqref{vn} and \eqref{vn1} in \cite{frankelSivashinsky1987} starts from a constant
density flame model, neglecting thermal expansion of the gas.   
Thus these are models exploring thermal-diffusive instability.  
This instability generates cellular structures which may be modeled with
free interface problems \cite{lunardi1}, \cite{lunardi2}, 
and models such as \eqref{vn} and \eqref{vn1} give the velocity of this interface.
In addition to \cite{frankelSivashinsky1987}, coordinate-free models for flame front propagation have been 
developed in \cite{frankel2003} and \cite{frankelSivashinsky1988}. Some analytical studies have been made
of these models, such as studying a quasi-steady problem \cite{roytburd}, \cite{braunerQS}.

The Kuramoto-Sivashinsky equation as given in \eqref{ff} is a form of the more general Kuramoto-Sivashinsky 
equation
\begin{equation}\label{generalKS}
\phi_{t}+\frac{1}{2}|\nabla\phi|^{2}=-c_{1}^{2}\Delta^{2}u-c_{2}^{2}\Delta u,
\end{equation}
in the case of one spatial dimension.  The two linear terms on the right-hand side play different roles, as
the fourth-order term is stabilizing and makes the problem well-posed, while the second-order term is 
destabilizing and can lead to growth of solutions.  The 
interaction of the nonlinear term on the left-hand side with the linear terms leads to rich and highly nontrivial 
dynamics, especially given the lack of a maximum principle for the equation owing to its fourth-order nature.
(We mention that there are versions of the coordinate-free models such as \eqref{vn} available in higher dimension
as well \cite{frankelSivashinsky1988}.)

The Kuramoto-Sivashinsky equation has been widely studied over the years, with global existence of solutions
and stability of the zero solution both established in one spatial dimension \cite{goodman}, \cite{NST}, \cite{TA}.
Detailed estimates have been developed in one spatial dimension for the dependence of the solutions on the size of the
periodic domain \cite{giacomelliOtto}, \cite{otto2}.
Many results for the Kuramoto-Sivashinsky equation in one spatial dimension rely on structure not present
in higher-dimensional problems, especially that an estimate for the $L^{2}$ norm of the first spatial derivative
of the unknown is available.  In higher dimensions this estimate is not available, and there are fewer results.
If the right-hand side of \eqref{generalKS} is modified to instead be $c_{1}^{2}\Delta u+c_{2}^{2}u,$ then a 
maximum principle is available and this structure may be used to find some global existence results \cite{goodman},
\cite{molinet}; the equation is then known instead as the Burgers-Sivashinsky equation.  Larios and 
Yamazaki have
also leveraged this structure for a system which blends features of the Kuramoto-Sivashinsky and Burgers-Sivashinsky
models \cite{larios}.  For the full Kuramoto-Sivashinsky equation in two spatial dimensions, 
Sell and Taboada have proven global existence of solutions in thin domains \cite{sell}, and the first author and
Mazzucato have shown global existence in the absence of linearly growing modes (which happens when the domain
is a sufficiently small torus) \cite{ambroseMazzucato}.  Additional results for the Kuramoto-Sivashinsky equation
on thin domains may be found in \cite{newThinResults} and \cite{molinet2}.

The distinction between known behavior in one spatial dimension and two spatial dimensions indicates
that the structures present in \eqref{ff} used to demonstrate, for example, global existence of solutions are 
perhaps a bit delicate and may not be present in closely related systems.  Indeed, while Frankel and Sivashinsky
have formally derived \eqref{ff} from the coordinate-free models \eqref{vn} and \eqref{vn1}, the authors are unaware
of any analytical theory for these relationships.  While the question of global existence of solutions for the
coordinate-free models remains open, we demonstrate short-time well-posedness here, focusing on
\eqref{vn1} for simplicity, and show rigorously the connection between solutions of \eqref{vn1} and \eqref{ff}.

There is a long history of demonstrating that weakly nonlinear models serve as valid approximations for more
fully nonlinear models; a key example of such work is the proof that the Korteweg-de Vries equation is a good
approximation of the irrotational Euler equations with a free surface \cite{kdv1}, \cite{kdv2}, \cite{kdv3}.  For more 
such works in the theory of water waves, the interested reader might consult the book of Lannes and the 
references therein \cite{lannes}.  While the Kuramoto-Sivashinsky equation is a widely studied weakly nonlinear model
for the propagation of flame fronts, the authors are unaware of any prior proofs of its validity in approximating more
highly nonlinear models.  The result in the literature most similar to the present work appears to be the main result  
of \cite{braunerEtAl}, in which solutions of the Kuramoto-Sivashinsky equation are shown to remain close to 
solutions of another weakly nonlinear model; this weakly nonlinear model is derived from coordinate-free models
similar to \eqref{vn}, but also incorporating temperature effects.

As we will first prove well-posedness of the initial value problem for the
coordinate-free model given by \eqref{vn1}, we first convert it into an evolutionary problem, which requires
setting coordinates.  We do so with an eye towards our approximation theorem, and so not making the most
general possible choice.  As the approximation theorem we prove is for the Kuramoto-Sivashinsky equation, and
the flame front in the Kuramoto-Sivashinsky equation is parameterized as a graph over the horizontal coordinate, $x,$
we thus make this choice of frame for the coordinate-free model.  We make the relevant calculations in the following
Section \ref{formulation}.  

This choice of restricting \eqref{vn1} to the case of a graph over the horizontal coordinate is not 
a limitation on our well-posedness theory; indeed it would be no more difficult to treat \eqref{vn1} for flame
fronts which could have multi-valued height or which might be closed curves.  To treat such scenarios, the 
parameterization of the curve could be set using tangent angle and arclength, as was done for 
interfaces between fluids in the numerical work of Hou, Lowengrub, and Shelley \cite{HLS1}, \cite{HLS2}.
The formulation of Hou, Lowengrub, and Shelley was subsequently used by the first author and collaborators
a number of times to prove well-posedness of initial value problems in interfacial fluid mechanics, for example in
the works \cite{ambroseThesis}, \cite{ambroseSiegel3}, \cite{liuAmbrose1}.  The advantage of the tangent angle
and arclength formulation is that these are naturally related to the curvature, and the curvature of the front
is what appears on the right-hand sides of \eqref{vn} and \eqref{vn1}.  The first author and Akers have implemented
numerical methods to compute the propagation of fronts using
the angle-arclength formulation for the models \eqref{vn} and \eqref{vn1} using further ideas from \cite{HLS1}
in the preprint \cite{akersAmbrose}.

\subsection{Reformulation: Setting coordinates}\label{formulation}
In order to compare the  equations \eqref{ff} with that of  \eqref{vn}, we need to have a more convenient form of the equation \eqref{vn}, i.e. change  the coordinates in \eqref{vn} from $(s, t)$ to $(x, t)$. Clearly we need to rewrite $V_n$ and $\kappa_{ss}$ in the new variables. \\

{\bf{Function $V_n$}:} For any curve $(x(\be, t), y(\be, \tau) )$ we can write the motion as a combination of the normal vector $n= \f{(y_{\be}, -x_{\be})}{|(y_{\be},- x_{\be})|}$ and the tangent vector $T= \f{(x_{\be}, y_{\be})}{|(x_{\be}, y_{\be})|}$. Furthermore, we have the following decomposition of $(x, y)_t$
\begin{equation}
(x,y)_t= V_n \cdot n+ V_{\tau} \cdot T,
\end{equation}
where $V_n$ is as it is defined above, and $V_{\tau}$ is related to the choice of the parameters.   As it is mentioned above, our model covers the case of $(x, y)= (x, f(x))$ and $x_t= 0$ (i.e. $x= \be$), therefore
$$
x_t= \f{ y_x V_n}{\sqrt{1+ y_x^2}}+ \f{V_{\tau}}{\sqrt{1+ y_x^2}}= 0 \Rightarrow V_{\tau}=-  y_x V_n.
$$
We can use the above to find $y_t$. Indeed,
\begin{eqnarray*}
y_t= \f{- V_n}{\sqrt{1+ y_x^2}}+ \f{y_x \cdot V_{\tau}}{\sqrt{1+ y_x^2}}= \f{- (1+ y_x^2) \cdot V_n}{\sqrt{1+ y_x^2}} =- \sqrt{1+ y_x^2} \cdot V_n.
\end{eqnarray*}
This clearly suggests that
\begin{eqnarray}\label{Vn}
V_n= \f{- y_t}{\sqrt{1+ y_x^2}}.
\end{eqnarray}
{\bf{Function $\kappa_{ss}$}:} Note that $\f{d s}{dx}= \sqrt{1+ y_x^2}$, therefore
\begin{eqnarray*}
\f{d \kappa}{d x}= \f{d \kappa}{ds} \cdot \f{d s}{d x} = \f{d \kappa}{ds} \cdot \sqrt{1+ y_x^2}
\end{eqnarray*} 
and consequently,
\begin{eqnarray*}
	\f{d^2 \kappa}{d x^2}&=& \f{d}{dx}\bigg(  \f{d \kappa}{ds} \cdot \sqrt{1+ y_x^2}\bigg)= \f{d^2 \kappa}{d s^2} \cdot \f{d s}{d x} \cdot \sqrt{1+ y_x^2}+ \f{d \kappa}{d s} \cdot \f{y_x y_{xx}}{\sqrt{1+ y_x^2}}\\
	&=& \f{d^2 \kappa}{d s^2}  \cdot (1+ y_x^2)+ \f{d \kappa}{d x} \cdot \f{y_x y_{xx}}{1+ y_x^2}
\end{eqnarray*}
In other words,
\begin{eqnarray}\label{ks}
\f{d^2 \kappa}{d s^2}= \f{1}{1+ y_x^2} \cdot \f{d^2 \kappa}{d x^2}- \f{y_ x y_{xx}}{(1+ y_x^2)^2} \cdot \f{d \kappa}{d x}
\end{eqnarray}
Now we insert \eqref{Vn} and \eqref{ks} into \eqref{vn} and get the following equation
\begin{eqnarray}\label{yt}
	\begin{cases}
y_t +  \f{(\al- 1) \cdot y_{xx}}{1+ y_x^2}+  \bigg(1+ \f{1}{2} \al^2\bigg) \f{y_{xx}^2}{(1+ y_x^2)^{\f{5}{2}}}+ \bigg(2 \al+ 5 \al^2- \f{1}{3} \al^3\bigg) \f{y_{xx}^3}{(1+ y_x^2)^4}+ \f{\al^2 (\al+ 3)}{\sqrt{1+ y_x^2}} \cdot \f{d^2 \kappa}{d x^2}+\\ 
\hspace{1em}+ \sqrt{1+ y_x^2}= \al^2 (\al+ 3)  y_x \cdot \kappa \cdot \f{d \kappa}{d x},\\
y(x, 0)= y_0(x).
	\end{cases}
\end{eqnarray}
where,
\begin{eqnarray*}
\f{d \kappa}{d x}&=& \f{y_{xxx}}{(1+ y_x^2)^{\f{3}{2}}}- \f{3 y_x \cdot (y_{xx})^2}{(1+ y_x^2)^{\f{5}{2}}},\\
\f{d^2 \kappa}{d x^2}&=& \f{y_{xxxx}}{(1+ y_x^2)^{\f{3}{2}}}- \f{ 3 (y_{xx})^3+ 9 y_x y_{xx} y_{xxx}}{(1+ y_x^2)^{\f{5}{2}}}+ \f{15 (y_x)^2 (y_{xx})^3 }{(1+ y_x^2)^{\f{7}{2}}}.
\end{eqnarray*}

In section \ref{sec:2.1} we recall some definitions, standard estimates from Harmonic analysis, as well as a form of Gronwall's inequality which fits our Grnonwall's type inequalities. In section \ref{sec:3} we present the existence of the solution of the equation \eqref{yt} in $H^4$.  In other words, section \ref{sec:3} covers the proof of Theorem \ref{thm1}. This is done via an approximate equation.   Finally, in section \ref{sec:5} we present a proof of Theorem \ref{thm2}. This is done via a coordinate scaling, where the scaling has been chosen carefully. 
\section{Preliminaries}

\subsection{Fourier series, function spaces and mulitpliers} 
\label{sec:2.1}   
We will consider periodic function spaces, although this is not essential.  A sufficiently regular 
function $f$ on a periodic interval
may be written with its Fourier series,
\begin{equation}\nonumber
f(x)=\sum_{p\in\mathbb{Z}}\hat{f}(p)e^{ipx}.
\end{equation}
Consequently, since $\widehat{-\De f}(p)= |p|^2 \hat{f}(p)$, we define the operators $|\nabla|^a:=(-\De)^{a/2}, a>0$, via its action on the Fourier side $\widehat{|\nabla|^a f}(p)= |p|^a \hat{f}(p)$.

The $L^p$ spaces are defined by the norm   $  \|f\|_{L^p}= \bigg( \int |f(x)|^p\ dx\bigg)^{\f{1}{p}}$. 
For $p\in (1, \infty)$,  the Sobolev spaces are the closure of the Schwartz functions in the norm $\|f\|_{W^{k, p}}= \|f\|_{L^p}+ \sum_{|\al| \leq k} \|\partial^{\al} f\|_{L^p}$,  while for a non-integer $s$ one takes   
\begin{equation*}
\|f\|_{W^{s, p}}= \|(1- \De)^{s/2} f\|_{L^p}\sim \|f\|_{L^p}+ \||\nabla|^s f\|_{L^p}.
\end{equation*}
The Sobolev embedding theorem states $\|f\|_{L^p(T^1)} \leq C \||\nabla|^s f\|_{L^q(T^1)}$, where $1<p<q<\infty$ and 
$ \f{1}{p}- \f{1}{q}= s$, with the usual modification for $p=\infty$, namely 
$\|f\|_{L^\infty(T^1)} \leq C_s \|  f\|_{W^{s,q}(T^1)}$, $s> \f{1}{p}$.  
Another useful ingredient will be  the Gagliardo - Nirenberg interpolation inequality,
\begin{equation*}
\| |\nabla|^s f\|_{L^p} \leq \||\nabla|^{s_1} f\|^{\theta}_{L^q} \||\nabla|^{s_2} f\|^{1- \theta}_{L^r},
\end{equation*}
where   $s= \theta s_1+ (1- \theta) s_2$  and $\f{1}{p}= \f{\theta}{q}+ \f{1- \theta}{r}$.   

Throughout this work we make use of a particular version of mollifier operators $\mathcal{J}^{\delta}$, $0< \delta<<1$, which represent the truncation of the Fourier series, zeroing out modes with wave number larger than $\frac{1}{\delta}$. We frequently use the following two essential properties of the mollifiers, which can be easily proved in a straightforward way using the Hausdorff-Young inequality, or alternatively the Plancherel theorem,  
\begin{eqnarray}\label{m0}
\|\mathcal{J}^{\delta} f\|_{H^s} &\leq& \|f\|_{H^s}\\ \label{m1}
\|\mathcal{J}^{\delta} \partial^s f\|_{L^2} &\leq& \frac{C}{\delta^s} \|f\|_{L^2}.
\end{eqnarray}
Note that the operator $\mathcal{J}^{\delta}$ is both a self-adjoint operator and a projection, i.e. $\mathcal{J}^{\delta} (\mathcal{J}^{\delta} f)= \mathcal{J}^{\delta} f$. Moreover,  it commutes with the derivative operator, $\mathcal{J}^{\delta} \partial f= \partial \mathcal{J}^{\delta} f$. 

\subsection{Gronwall's inequality}
 We need the following two versions of the Gronwall's inequality:
\begin{lem}\label{gr}
	Let the functions $x, a, b,$ and $k$ be continuous and nonnegative
	on the interval 
	$J = [\al, \be]$, and let $n$ be a positive integer $(n \geq 2).$ Assume $\frac{a}{b}$ 	
	is a nondecreasing function. If
	\begin{eqnarray}
	x(t) \leq a(t)+ b(t) \int_\al^t k(s) x^n(s) ds,\ \ \ t \in J,
	\end{eqnarray}
	then 
	\begin{equation}
	x(t) \leq a(t) \bigg\{1- (n-1) \int_{\al}^t k(s) b(s) a^{n-1}(s) ds \bigg\}^{\frac{1}{n- 1}},\ \ \ \al \leq t \leq \be_n,
	\end{equation}
	where $\beta_{n}$ is given by
	\begin{eqnarray}\label{tim}
	\be_n= \sup \bigg\{ t \in J: (n-1) \int_{\al}^t k(s) b(s) a^{n-1}(s) ds < 1 \bigg\}.
	\end{eqnarray}
	
\end{lem}

\begin{lem}\label{grr}
	Fix $\tau_*$ and $\Ga_*> 0$. Assume the function $E(t)$ satisfies the relation 
	 \begin{eqnarray}
	\f{d}{d t} E(t) \leq \al E(t)+ \be E^2(t)+ \eps^{n} \bigg(E(t) \bigg)^{m},
	\end{eqnarray}
	where $0< \eps<< 1$, $n \geq 0$, and $m\geq 1$.  Then there exists $E_*$ and $\eps_*$ so that for any $E(0)= E_0 \leq E_*$ and $0< \eps \leq \eps_*$
	\begin{eqnarray}
	\sup_{0< \tau< \tau_*} |E(t)| \leq \Ga_*.
	\end{eqnarray}
\end{lem}
Both of these versions of Gronwall's inequality are known. Here we give our own proof for Lemma \ref{grr}. One can find the proof of Lemma \ref{gr} in  \cite[Theorem 25]{SSD}.
\begin{proof}
	In order to prove Lemma \ref{grr} fix $\Ga_*$, and let $E(t_0)$ be the first time at which $E(t_0)= \Ga_*$ (if for all $t> 0$, $E(t_0)< \Ga_*$ then let $t_0= \infty$, in which case the proof is completed). Hence, for any $t \in [0, t_0]$ we have $E^m \leq \Ga_*^{m-1} E$. Therefore,
	\begin{eqnarray}
	\f{d}{d t} E(t) \leq \bigg(\al + \be \Ga_*+ \eps^{n} \Ga_*^{m- 1} \bigg) E(t).
	\end{eqnarray}
	Now we apply the routine Gronwall's inequality to this relation, and we get, for any $t \in [0, t_0]$
	\begin{eqnarray}
	E(t) \leq \exp\bigg((\al + \be \Ga_*+ \eps^{n} \Ga_*^{m- 1} ) t\bigg) E_0.
	\end{eqnarray}
	At $t= t_0$, we have $E(t_0)= \Ga_*$, hence
	\begin{eqnarray*}
	\Ga_* \leq \exp\bigg((\al + \be \Ga_*+ \eps^{n} \Ga_*^{m- 1} ) t_0\bigg) E_0,
	\end{eqnarray*}
which implies 
\begin{eqnarray*}
 t_0 \geq \f{\ln\left(\frac{\Ga_*}{E_0}\right)}{\al + \be \Ga_*+ \eps^{n} \Ga_*^{m- 1}} =:
 \tau_0(\Ga_*,E_0,\eps).
\end{eqnarray*} 
Note that $\tau_0(\Ga_*, E_0, \eps)$ is decreasing with $\eps$ and with $E_0$. What we have shown so far asserts that if $0 \leq t \leq \tau_0(\Ga_*, E_0, \eps)$, then
\begin{eqnarray}\label{ee}
E(t) \leq \Ga_*.
\end{eqnarray}
Now fix a time $t_*$, and $\Ga_*$ as well as $\eps \leq 1:= \eps_*$, and solve $\tau_0(\Ga_*, E_0, \eps)= t_*$ for $E_*$,
namely
\begin{eqnarray}
E_*= \Ga_* \exp\bigg( (\al + \be \Ga_*+ \eps^{n} \Ga_*^{m- 1}) t_*\bigg).
\end{eqnarray}
Now we claim that with $t_*, \Ga_*$ and $E_*$ as above, then if $E_0 \leq E_*$ and $\eps< 1$ we have 
\begin{eqnarray}
\sup_{0< \tau< \tau_*} |E(t)| \leq \Ga_*.
\end{eqnarray}
Indeed, by \eqref{ee} we have $E(t) \leq \Ga_*$ for $0 \leq t \leq \tau_0(\Ga_*, E_0, \eps)$. Since $\tau_0(\Ga_*, E_0, \eps)$ is decreasing with respect to $E_0$ and $\eps$, we know
\begin{eqnarray*}
t_*= \tau_0(\Ga_*, E_*, 1) \leq \tau_0(\Ga_*, E_0, \eps).
\end{eqnarray*}
Thus
\begin{eqnarray*}
\{t: 0 \leq t \leq t_*\} \subset \{t: 0 \leq t \leq \tau_0(\Ga_*, E_0, \eps)\},
\end{eqnarray*}
and we get 
\begin{eqnarray}
\sup_{0< \tau< \tau_*} |E(t)| \leq \Ga_*.
\end{eqnarray}
\end{proof}

\subsection{Main Result}
As it is mentioned before we pursue two main goals in this article. First we aim to prove the well-posedness of the initial value
problem associated to 
\eqref{yt}. This is the content of Theorem \ref{thm1}. Our second goal is to show that the solution to equation \eqref{yt} stays close enough to the solution of the equation \eqref{ff}, in a sense to be made precise. In Theorem \ref{thm2} we present the related result.
\begin{thm}\label{thm1}
	Let $y(0) \in H^{5}$ be given. Then there exists a time $T= T(\|y(0)\|_{H^4})$ and a function $y \in C([0, T], H^5)$ which 
	satisfies \eqref{yt}, and the initial condition $y (\cdot, 0)= y(0)$.
\end{thm}
\begin{thm}\label{thm2}
	Fix $\tau_{*}>0$ and $\Ga_{*}> 0.$ Then there exists $\eps_*$ and $E_*$ so that whenever $0< \eps< \eps_*$ and $\|U_0(\cdot)\|_{H^4} \leq E_{*},$ the following hold:
	
	Let $y(x, t)$ be the solution of \eqref{yt} with $\al- 1= \eps$, and 
	\begin{eqnarray}\label{In}
	y(x, 0)= \eps U_0(\sqrt{\eps} x).
	\end{eqnarray}
	Let $U(\xi, \tau)$ be the solution of  the Kuramoto-Sivashinsky equation
	\begin{eqnarray}\label{UU}
	\partial_{\tau} U+ \f{1}{2} (\partial_{\xi} U)^2+ \partial_{\xi}^2 U+  4 \partial_{\xi}^4 U= 0,
	\end{eqnarray}
	with $U(\xi, 0)= U_0(\xi)$. Then
		\begin{eqnarray}\label{main}
	\sup_{0< t< \f{\tau_*}{\eps^2}} \|y(\cdot, t)+ t-  \eps U(\sqrt{\eps}\ \cdot, \eps^2 t)\|_{L^2} \leq \Ga_* \eps^{\f{7}{4}}.
	\end{eqnarray}
\end{thm}

The proofs of Theorems \ref{thm1} and \ref{thm2} are presented in Lemma \ref{lem0} and Remark \ref{remproof}, respectively.
\begin{rem}
	For simplicity in our calculations, we choose the initial data of the equation \eqref{UU} to be \eqref{In}. Our proofs, however, indicate that any other initial data close enough to $\f{1}{\eps} y_0(\f{x}{\eps})$ leads to the same result. See the proof of Lemma \ref{PhiU}. 
\end{rem}
\begin{rem}
	The time interval presented in Theorem \ref{thm1}  increases  for a smaller  $\|y(0)\|_{H^4}$. In fact  $T< C \ln\bigg(1+ \f{C}{\|y_0\|^{m- 2}_{H^4}}\bigg)$, for some positive $m$  to be defined in the sequel. 
\end{rem}

\section{Existence of the solution}\label{sec:3}
The first step toward the completion of the argument is to show that the equation  \eqref{yt} has an unique solution in some Sobolev spaces, over a time interval $[0, T]$, with $T$ to be determined. The proof follows the energy method. To that end, we first introduce approximate equations, where the approximation are introduced via a multiplier operator $\mathcal{J}^{\de}$. We next use the Picard Theorem to find
that the approximate equations admit unique solutions in some Sobolev spaces over a time interval $[0, T_{\de}]$. This $T_{\de}$ might be small (i.e., this time depends badly on the approximation parameter $\de$). Therefore, in an attempt to increase $T_{\de}$, we prove
 bounds on the solution which are uniform with respect to $\de.$ Once the uniform bounds are in hand, since norms of the solutions of the approximate equations are not increasing fast,  the solutions may be continued to a time interval $[0, T]$, where $T$ can be taken
 to be independent of $\de.$  Finally, with solutions existing on a uniform time interval, the limit may be taken as $\de$ vanishes, and this
 limit can be seen to satisfy the correct initial value problem. 

We define $y^{\delta}$ to be the solution of the following initial value problem:
\begin{equation}\label{yt0}
\begin{cases}
y^{\delta}_t +  (\al- 1) \mathcal{J}^{\delta} \bigg[\f{ \mathcal{J}^{\delta} y^{\delta}_{xx}}{1+ (\mathcal{J}^{\delta} y^{\delta}_x)^2}\bigg]+ \bigg(1+ \f{1}{2} \al^2\bigg) \mathcal{J}^{\delta} \bigg[\f{(\mathcal{J} y^{\delta}_{xx})^2}{(1+ (\mathcal{J}^{\delta} y^{\delta}_x)^2)^{\f{5}{2}}} \bigg]+  \bigg(2 \al+ 5 \al^2- \f{1}{3} \al^3\bigg) \mathcal{J}^{\delta} \bigg[\f{(\mathcal{J} y^{\delta}_{xx})^3}{(1+ (\mathcal{J}^{\delta} y^{\delta}_x)^2)^4} \bigg]+ \\
\hspace{2em}+\al^2 (\al+ 3)  \mathcal{J}^{\delta} \bigg[\f{1}{\sqrt{1+ (\mathcal{J}^{\delta} y^{\delta}_x)^2}} \cdot\f{d^2 \kappa^{\delta}}{d x^2}\bigg]
+   \mathcal{J}^{\delta} \bigg[\sqrt{1+ (\mathcal{J}^{\delta} y^{\delta}_x)^2}\bigg]=  \al^2 (\al+ 3)  \mathcal{J}^{\delta} \bigg[(\mathcal{J}^{\delta} y^{\delta}_x) \cdot \kappa^{\delta} \cdot \f{d \kappa^{\delta}}{d x}\bigg],\\
y^{\delta}(x, 0)=\mathcal{J}^{\delta}  y_0(x),
\end{cases}
\end{equation}
where
\begin{eqnarray}
	\kappa^{\delta}&=& \f{\mathcal{J}^{\delta} y^{\delta}_{xx}}{(1+ (\mathcal{J}^{\delta} y^{\delta}_x)^2)^{\f{3}{2}}},\\
\label{oneDerivativeKappa}	\f{d \kappa^{\delta}}{d x}&=& \f{\mathcal{J}^{\delta} y^{\delta}_{xxx}}{(1+ (\mathcal{J}^{\delta} y^{\delta}_x)^2)^{\f{3}{2}}}-3  \f{ (\mathcal{J}^{\delta} y^{\delta}_x) \cdot (\mathcal{J}^{\delta} y^{\delta}_{xx})^2}{(1+ (\mathcal{J}^{\delta} y^{\delta}_x)^2)^{\f{5}{2}}},\\
\label{twoDerivativesKappa}	\f{d^2 \kappa^{\delta}}{d x^2}&=& \f{\mathcal{J}^{\delta} y^{\delta}_{xxxx}}{(1+ (\mathcal{J}^{\delta} y^{\delta}_x)^2)^{\f{3}{2}}}- \f{ 3 (\mathcal{J}^{\delta} y^{\delta}_{xx})^3+ 9 (\mathcal{J}^{\delta} y^{\delta}_x) (\mathcal{J}^{\delta} y^{\delta}_{xx}) (\mathcal{J}^{\delta} y^{\delta}_{xxx})}{(1+ (\mathcal{J}^{\delta} y^{\delta}_x)^2)^{\f{5}{2}}}+ \f{15 ( \mathcal{J}^{\delta} y^{\delta}_x)^2 (\mathcal{J}^{\delta} y^{\delta}_{xx})^3 }{(1+ (\mathcal{J}^{\delta} y^{\delta}_x)^2)^{\f{7}{2}}}.
\end{eqnarray} 
We now present the first step toward the existence argument.  We show that the equation \eqref{yt0} admits a solution up to a small time $T_{\de}$. 
\begin{lem}
	Let $y(0) \in H^5$ be given. For any $\delta> 0$, for any $s\geq0,$ there is a time $T_{\delta}$ and a function $y^{\delta} \in C^1([0, T_\delta], H^{s})$ that satisfies \eqref{yt0}, as well as $y^{\delta} (\cdot, 0)= \mathcal{J}^{\delta}y(0)$.
\end{lem}
\begin{proof}
Since the initial data is mollified, it is in any Sobolev space.  With the abundance of mollifiers present on the right-hand side of the
evolution equation, it is not difficult to demonstrate that the relevant operator is a Lipschitz map.  The Picard Theorem applies, 
leading to the conclusion of the theorem.  We omit further details. 
\end{proof}

The next two lemmas concern some uniform bounds on the solution of the equation \eqref{yt0}. In the first lemma we prove an $H^4$ bound, and  we then use it in the subsequent lemma for a $H^5$ bound.
\begin{lem}\label{lem00}
	Assume $y^{\delta}$ is the solution of the equation \eqref{yt}. Then there exists $T= T(\al)$ and $C= C(y_0, \al)$, independent of $\delta$, so that for any  $0< t<   \f{\ln(1+ \f{\ga}{\|y_0\|^{m- 2}_{H^4}})}{\ga}$ ($m$ and $\ga$ to be defined later),
	\begin{eqnarray}
	\sup_{0< t< T} \|y^{\delta}\|^2_{H^4}+   \int  \f{ (\mathcal{J}^{\delta} y^{\delta}_{xx})^2+ (\partial_x^6 \mathcal{J}^{\delta} y^{\delta})^2}{(1+ (\mathcal{J}^{\delta} y^{\delta}_x)^2)^2}   dx \leq C
	\end{eqnarray}
\end{lem}

\begin{proof}
	During the proof, we assume that $\|y^{\delta}\|^2_{L^2}+ \|\partial^4_x y^{\delta}\|^2_{L^2} > 1$, otherwise  there is nothing  to prove. 
	
	In order to prove this lemma, we combine two energy estimates, one on $\|y^{\de}\|_{L^2}$, and the other one on $\|\partial^4_x y^{\de}\|_{L^2}$. Indeed,	
	\begin{eqnarray}\label{en00}
	&&\f{1}{2} \partial_t \|y^{\delta}\|^2_{L^2}+ (\al- 1) \int (\mathcal{J}^{\delta} y^{\delta}) \cdot  \bigg[\f{ (\mathcal{J}^{\delta} y^{\delta}_{xx})}{1+ (\mathcal{J}^{\delta} y^{\delta}_x)^2}\bigg] dx\\ \nonumber
	&&+ \bigg(1+ \f{1}{2} \al^2\bigg)  \int (\mathcal{J}^{\delta} y^{\delta}) \cdot\mathcal{J}^{\delta} \bigg[\f{(\mathcal{J} y^{\delta}_{xx})^2}{(1+ (\mathcal{J}^{\delta} y^{\delta}_x)^2)^{\f{5}{2}}} \bigg] dx+ \al^2 (\al+ 3) \int  (\mathcal{J}^{\delta} y^{\delta}) \cdot  \bigg[\f{1}{\sqrt{1+ (\mathcal{J}^{\delta} y^{\delta}_x)^2}} \cdot\f{d^2 \kappa^{\delta}}{d x^2}\bigg] dx\\ \nonumber
	&&+ \int (\mathcal{J}^{\delta} y^{\delta}) \cdot  \bigg[\sqrt{1+ (\mathcal{J}^{\delta} \mathcal{J}^{\delta} y^{\delta}_x)^2}\bigg] dx+  \bigg(2 \al+ 5 \al^2- \f{1}{3} \al^3\bigg)  \int (\mathcal{J}^{\delta} y^{\delta}) \cdot  \bigg[\f{(\mathcal{J} y^{\delta}_{xx})^3}{(1+ (\mathcal{J}^{\delta} y^{\delta}_x)^2)^4} \bigg] dx\\ \nonumber
	&& = \al^2 (\al+ 3) \int (\mathcal{J}^{\delta} y^{\delta}) \cdot  \bigg[(\mathcal{J}^{\delta}  y^{\delta}_x) \cdot \kappa^{\delta} \cdot \f{d \kappa^{\delta}}{d x}\bigg] dx.
	\end{eqnarray} 
We use integration by parts to arrive at a more convenient form for this expression.

The first term we simplify produces a useful term in the left hand side of \eqref{en00}, namely $\int  \f{ (\mathcal{J}^{\delta} y^{\delta}_{xx})^2}{(1+ (\mathcal{J}^{\delta} y^{\delta}_x)^2)^2}   dx$. Indeed, when we substitute from
\eqref{twoDerivativesKappa} into the fourth term on the left-hand side of \eqref{en00}, we find 
		\begin{eqnarray*}
		&&\al^2 (\al+ 3) \int  ( \mathcal{J}^{\delta} y^{\delta}) \cdot  \bigg[\f{1}{\sqrt{1+ (\mathcal{J}^{\delta} y^{\delta}_x)^2}} \cdot\f{d^2 \kappa^{\delta}}{d x^2}\bigg]\ dx= \al^2 (\al+ 3) \int  \f{(\mathcal{J}^{\delta} y^{\delta}) \cdot (\mathcal{J}^{\delta} y^{\delta}_{xxxx})}{(1+ (\mathcal{J}^{\delta} y^{\delta}_x)^2)^2}\ dx\\
		&&- 3 \al^2 (\al+ 3) \int \f{ (\mathcal{J}^{\delta} y^{\delta}) \cdot (\mathcal{J}^{\delta} y^{\delta}_{xx})^3}{(1+ (\mathcal{J}^{\delta} y^{\delta}_x)^2)^3}\ dx-  9 \al^2 (\al+ 3) \int \f{ (\mathcal{J}^{\delta} y^{\delta})  (\mathcal{J}^{\delta} y^{\delta}_x) (\mathcal{J}^{\delta} y^{\delta}_{xx}) (\mathcal{J}^{\delta} y^{\delta}_{xxx})}{(1+ (\mathcal{J}^{\delta} y^{\delta}_x)^2)^3}\ dx\\
		&&+ 15 \al^2 (\al+ 3) \int \f{ ( \mathcal{J}^{\delta} y^{\delta}) \cdot (\mathcal{J}^{\delta} y^{\delta}_x)^2 (\mathcal{J}^{\delta} y^{\delta}_{xx})^3 }{(1+ (\mathcal{J}^{\delta} y^{\delta}_x)^2)^4}\ dx.
	\end{eqnarray*}
The term we wish to draw out can now be found after integrating by parts twice:
	\begin{eqnarray*}
	&& \int  \f{(\mathcal{J}^{\delta} y^{\delta}) \cdot (\mathcal{J}^{\delta} y^{\delta}_{xxxx})}{(1+ (\mathcal{J}^{\delta} y^{\delta}_x)^2)^2}= -  \int \f{(\mathcal{J}^{\delta} y^{\delta}_{xxx}) \cdot (\mathcal{J}^{\delta} y^{\delta}_x)}{(1+ (\mathcal{J}^{\delta} y^{\delta}_{x})^2)^2} dx+ 4 \int \f{(\mathcal{J}^{\delta} y^{\delta}) \cdot (\mathcal{J}^{\delta} y^{\delta}_x) \cdot (\mathcal{J}^{\delta} y^{\delta}_{xx}) (\mathcal{J}^{\delta} y^{\delta}_{xxx})}{(1+ (\mathcal{J}^{\delta} y^{\delta}_x)^2)^3} dx\\
	&&=  \int \f{(\mathcal{J}^{\delta} y^{\delta}_{xx})^2}{(1+ (\mathcal{J}^{\delta} y^{\delta}_{x})^2)^2} dx- 4 \int \f{ (\mathcal{J}^{\delta} y^{\delta}_x)^2 \cdot (\mathcal{J}^{\delta} y^{\delta}_{xx})^2}{(1+ (\mathcal{J}^{\delta} y^{\delta}_{x})^2)^3} dx + 4 \int \f{(\mathcal{J}^{\delta} y^{\delta}) \cdot (\mathcal{J}^{\delta} y^{\delta}_x) \cdot (\mathcal{J}^{\delta} y^{\delta}_{xx}) (\mathcal{J}^{\delta} y^{\delta}_{xxx})}{(1+ (\mathcal{J}^{\delta} y^{\delta}_x)^2)^3} dx.
\end{eqnarray*}
Our conclusion is
	\begin{eqnarray*}
	&&\al^2 (\al+ 3) \int  (\mathcal{J}^{\delta} y^{\delta}) \cdot  \bigg[\f{1}{\sqrt{1+ (\mathcal{J}^{\delta} y^{\delta}_x)^2}} \cdot\f{d^2 \kappa^{\delta}}{d x^2}\bigg] dx= \al^2 (\al+ 3) \int \f{(\mathcal{J}^{\delta} y^{\delta}_{xx})^2 }{(1+ (\mathcal{J}^{\delta} y^{\delta}_{x})^2)^2} dx\\
	&&- 4 \al^2 (\al+ 3) \int \f{ (\mathcal{J}^{\delta} y^{\delta}_x)^2 \cdot (\mathcal{J}^{\delta} y^{\delta}_{xx})^2}{(1+ (\mathcal{J}^{\delta} y^{\delta}_{x})^2)^3} dx  - 3 \al^2 (\al+ 3) \int \f{(\mathcal{J}^{\delta} y^{\delta}) \cdot (\mathcal{J}^{\delta} y^{\delta}_{xx})^3 }{(1+ (\mathcal{J}^{\delta} y^{\delta}_{x})^2)^3} dx\\
	&&- 5 \al^2 (\al+ 3)  \int \f{(\mathcal{J}^{\delta} y^{\delta}) \cdot (\mathcal{J}^{\delta} y^{\delta}_x) \cdot (\mathcal{J}^{\delta} y^{\delta}_{xx}) (\mathcal{J}^{\delta} y^{\delta}_{xxx})}{(1+ (\mathcal{J}^{\delta} y^{\delta}_x)^2)^3} dx
+ 15 \al^2 (\al+ 3) \int \f{(\mathcal{J}^{\delta} y^{\delta}) \cdot (\mathcal{J}^{\delta} y^{\delta}_x)^2 \cdot (\mathcal{J}^{\delta} y^{\delta}_{xx})^3 }{(1+ (\mathcal{J}^{\delta} y^{\delta}_x)^2)^4} dx.\\
\end{eqnarray*}

For the right-hand side of \eqref{en00}, we substitute from \eqref{oneDerivativeKappa}, finding
	\begin{eqnarray*}
	\al^2 (\al+ 3) \int (\mathcal{J}^{\delta} y^{\delta}) \cdot  \bigg[( \mathcal{J}^{\delta} y^{\delta}_x)  \cdot \kappa^{\delta} \cdot \f{d \kappa^{\delta}}{d x}\bigg] dx&=&  \al^2 (\al+ 3)  \int \f{( \mathcal{J}^{\delta} y^{\delta}) \cdot ( \mathcal{J}^{\delta} y^{\delta}_x)  \cdot( \mathcal{J}^{\delta} y^{\delta}_{xx})( \mathcal{J}^{\delta} y^{\delta}_{xxx})}{(1+ (\mathcal{J}^{\delta} y^{\delta}_x)^2)^3} dx
	\\
	&-& 3 \al^2 (\al+ 3) \int \f{( \mathcal{J}^{\delta} y^{\delta}) \cdot (\mathcal{J}^{\delta} y^{\delta}_x)^2 \cdot (\mathcal{J}^{\delta} y^{\delta}_{xx})^3 }{(1+ (\mathcal{J}^{\delta} y^{\delta}_x)^2)^4} dx.
\end{eqnarray*}

We also rewrite the fifth term on the left-hand side of \eqref{en00} as
\begin{eqnarray*}
\int ( \mathcal{J}^{\delta} y^{\delta})\sqrt{1+ (\mathcal{J}^{\delta} y^{\delta}_x)^2}\ dx= \int \f{( \mathcal{J}^{\delta} y^{\delta})\bigg(1+ (\mathcal{J}^{\delta} y^{\delta}_x)^2\bigg)}{\sqrt{1+ (\mathcal{J}^{\delta} y^{\delta}_x)^2}}\ dx.
\end{eqnarray*}

With all of these considerations, \eqref{en00} now may be written as
	\begin{eqnarray}\label{en10}
&&\f{1}{2} \partial_t \|y^{\delta}\|^2_{L^2}+ \al^2 (\al+ 3) \int \f{(\mathcal{J}^{\delta} y^{\delta}_{xx})^2}{(1+ (\mathcal{J}^{\delta} y^{\delta}_{x})^2)^2} dx= - (\al- 1) \int \f{( \mathcal{J}^{\delta} y^{\delta})\cdot( \mathcal{J}^{\delta} y^{\delta}_{xx})}{(1+ (\mathcal{J}^{\delta} y^{\delta}_x)^2)} dx+\\ \nonumber
&& - \int \f{( \mathcal{J}^{\delta} y^{\delta})\bigg(1+ (\mathcal{J}^{\delta} y^{\delta}_x)^2\bigg)}{\sqrt{1+ (\mathcal{J}^{\delta} y^{\delta}_x)^2}} dx+  3 \al^2 (\al+ 3) \int \f{( \mathcal{J}^{\delta} y^{\delta})\cdot (\mathcal{J}^{\delta} y^{\delta}_{xx})^3 }{(1+ (\mathcal{J}^{\delta} y^{\delta}_{x})^2)^3} dx\\ \nonumber
&&+ 4 \al^2 (\al+ 3) \int \f{ (\mathcal{J}^{\delta} y^{\delta}_x)^2 \cdot (\mathcal{J}^{\delta} y^{\delta}_{xx})^2}{(1+ (\mathcal{J}^{\delta} y^{\delta}_{x})^2)^3} dx+
6 \al^2 (\al+ 3) \int \f{( \mathcal{J}^{\delta} y^{\delta})\cdot ( \mathcal{J}^{\delta} y^{\delta}_x)  \cdot( \mathcal{J}^{\delta} y^{\delta}_{xx})( \mathcal{J}^{\delta} y^{\delta}_{xxx})}{(1+ (\mathcal{J}^{\delta} y^{\delta}_x)^2)^3} dx
\\ \nonumber
&&- 15 \al^2 (\al+ 3) \int \f{( \mathcal{J}^{\delta} y^{\delta})\cdot (\mathcal{J}^{\delta} y^{\delta}_x)^2 \cdot (\mathcal{J}^{\delta} y^{\delta}_{xx})^3 }{(1+ (\mathcal{J}^{\delta} y^{\delta}_{xx})^2)^4} dx-  \bigg(1+ \f{1}{2} \al^2\bigg)  \int  \bigg[\f{(\mathcal{J}^{\delta} y^{\delta}) \cdot (\mathcal{J} y^{\delta}_{xx})^2}{(1+ (\mathcal{J}^{\delta} y^{\delta}_x)^2)^{\f{5}{2}}} \bigg] dx\\ \nonumber
&&+ \bigg(2 \al+ 5 \al^2- \f{1}{3} \al^3\bigg)  \int   \bigg[\f{(\mathcal{J}^{\delta} y^{\delta}) \cdot (\mathcal{J} y^{\delta}_{xx})^3}{(1+ (\mathcal{J}^{\delta} y^{\delta}_x)^2)^4} \bigg] dx.
\end{eqnarray}
All the terms on the right hand side are controlled by terms of 
the form of $C \bigg( \|y^{\delta}\|^a_{L^2}+ \|\partial^4_x y^{\delta}\|^a_{L^2}\bigg) $, where $2 \leq a \leq 4$. 
Overall, we have the following  simplified inequality:
\begin{eqnarray}\label{en30}
\f{1}{2} \partial_t \|y^{\delta}\|^2_{L^2}+ \al^2 (\al+ 3) \int \f{(\mathcal{J}^{\delta} y^{\delta}_{xx})^2}{(1+ (\mathcal{J}^{\delta} y^{\delta}_{x})^2)^2} dx \leq  C \bigg( \|y^{\delta}\|^2_{L^2}+ \|\partial^4_x y^{\delta}\|^2_{L^2}\bigg)+ \bigg( \|y^{\delta}\|^4_{L^2}+ \|\partial^4_x y^{\delta}\|^4_{L^2}\bigg).
\end{eqnarray}
This is straightforward to see (it mainly consists of counting derivatives) and we omit further details of the proof 
of \eqref{en30}.

We now turn our attention to the rest of the energy estimate. We take four spatial derivatives of \eqref{yt0}, 
and then find its inner product with $\partial^4_x y^{\delta}:$
\begin{eqnarray}\label{es0}
	&&\hspace{3em}\f{1}{2} \partial_t \|\partial_x^4 y^{\delta}\|^2_{L^2}+ (\al- 1) \int (\mathcal{J}^{\delta} \partial_x^4 y^{\delta}) \cdot \partial_x^4 \bigg[\f{ (\mathcal{J}^{\delta} y^{\delta}_{xx})}{1+ (\mathcal{J}^{\delta} y^{\delta}_x)^2}\bigg] dx+\\ \nonumber
&&+ \bigg(1+ \f{1}{2} \al^2\bigg)  \int (\mathcal{J}^{\delta} \partial_x^4 y^{\delta}) \cdot \partial_x^4 \bigg[\f{(\mathcal{J} y^{\delta}_{xx})^2}{(1+ (\mathcal{J}^{\delta} y^{\delta}_x)^2)^{\f{5}{2}}} \bigg] dx+ \al^2 (\al+ 3) \int  (\mathcal{J}^{\delta} \partial_x^4 y^{\delta}) \cdot \partial_x^4 \bigg[\f{1}{\sqrt{1+ (\mathcal{J}^{\delta} y^{\delta}_x)^2}} \cdot\f{d^2 \kappa^{\delta}}{d x^2}\bigg] dx\\ \nonumber
&&+ \int (\mathcal{J}^{\delta} \partial_x^4 y^{\delta}) \cdot \partial_x^4 \bigg[\sqrt{1+ ( \mathcal{J}^{\delta} y^{\delta}_x)^2}\bigg] dx+  \bigg(2 \al+ 5 \al^2- \f{1}{3} \al^3\bigg)  \int (\mathcal{J}^{\delta} \partial_x^4 y^{\delta}) \cdot \partial_x^4 \bigg[\f{(\mathcal{J} y^{\delta}_{xx})^3}{(1+ (\mathcal{J}^{\delta} y^{\delta}_x)^2)^4} \bigg] dx\\ \nonumber
&& = \al^2 (\al+ 3) \int (\mathcal{J}^{\delta} \partial_x^4 y^{\delta}) \cdot  \partial_x^4 \bigg[(\mathcal{J}^{\delta}  y^{\delta}_x) \cdot \kappa^{\delta} \cdot \f{d \kappa^{\delta}}{d x}\bigg] dx.
\end{eqnarray}

As before, for the fourth term on the left-hand side of \eqref{es0}, we substitute from \eqref{twoDerivativesKappa}:
\begin{eqnarray*}
&&\al^2 (\al+ 3) \int  (\mathcal{J}^{\delta} \partial_x^4 y^{\delta}) \cdot \partial_x^4 \bigg[\f{1}{\sqrt{1+ (\mathcal{J}^{\delta} y^{\delta}_x)^2}} \cdot\f{d^2 \kappa^{\delta}}{d x^2}\bigg] dx=  \al^2 (\al+ 3) \int  (\mathcal{J}^{\delta} \partial_x^6 y^{\delta}) \cdot \partial_x^2 \bigg[ \f{\mathcal{J}^{\delta} y^{\delta}_{xxxx}}{(1+ (\mathcal{J}^{\delta} y^{\delta}_x)^2)^2}  \bigg] dx\\
&&- 3 \al^2 (\al+ 3) \int  (\mathcal{J}^{\delta} \partial_x^6 y^{\delta}) \cdot \partial_x^2 \bigg[ \f{  (\mathcal{J}^{\delta} y^{\delta}_{xx})^3}{(1+ (\mathcal{J}^{\delta} y^{\delta}_x)^2)^3} \bigg] dx-  9 \al^2 (\al+ 3) \int  (\mathcal{J}^{\delta} \partial_x^6 y^{\delta}) \cdot \partial_x^2 \bigg[ \f{  (\mathcal{J}^{\delta} y^{\delta}_x) (\mathcal{J}^{\delta} y^{\delta}_{xx}) (\mathcal{J}^{\delta} y^{\delta}_{xxx})}{(1+ (\mathcal{J}^{\delta} y^{\delta}_x)^2)^3} \bigg] dx\\
&&+ 15 \al^2 (\al+ 3) \int  (\mathcal{J}^{\delta} \partial_x^6 y^{\delta}) \cdot \partial_x^2 \bigg[ \f{ ( \mathcal{J}^{\delta} y^{\delta}_x)^2 (\mathcal{J}^{\delta} y^{\delta}_{xx})^3 }{(1+ (\mathcal{J}^{\delta} y^{\delta}_x)^2)^4} \bigg] dx.
\end{eqnarray*}
We expand the first term on the right hand side of the above equality as follows:
\begin{eqnarray*}
 &&\al^2 (\al+ 3) \int  (\mathcal{J}^{\delta} \partial_x^4 y^{\delta}) \cdot \partial_x^4 \bigg[ \f{\mathcal{J}^{\delta} y^{\delta}_{xxxx}}{(1+ (\mathcal{J}^{\delta} y^{\delta}_x)^2)^2}  \bigg] dx=  \al^2 (\al+ 3) \int    \f{(\mathcal{J}^{\delta} \partial_x^6 y^{\delta})^2}{(1+ (\mathcal{J}^{\delta} y^{\delta}_x)^2)^2} dx\\
 &&- 4 \al^2 (\al+ 3) \int (\mathcal{J}^{\delta} \partial_x^6 y^{\delta}) \bigg[\f{(\mathcal{J}^{\delta} y^{\delta}_x) (\mathcal{J}^{\delta} y^{\delta}_{xx}) (\mathcal{J}^{\delta} \partial^5 y^{\delta})}{(1+ (\mathcal{J}^{\delta} y^{\delta}_x)^2)^3} \bigg] dx\\
 &&- 4 \al^2 (\al+ 3) \int (\mathcal{J}^{\delta} \partial_x^6 y^{\delta}) \partial \bigg[\f{ (\mathcal{J}^{\delta} y^{\delta}_x) (\mathcal{J}^{\delta} y^{\delta}_{xx}) (\mathcal{J}^{\delta} \partial^4 y^{\delta})}{(1+ (\mathcal{J}^{\delta} y^{\delta}_x)^2)^3} \bigg] dx.
\end{eqnarray*}

Integrating by parts twice, and using \eqref{oneDerivativeKappa}, we also have the formula
\begin{eqnarray*}
&&\al^2 (\al+ 3) \int (\mathcal{J}^{\delta} \partial_x^4 y^{\delta}) \cdot  \partial_x^4 \bigg[(\mathcal{J}^{\delta}  y^{\delta}_x) \cdot \kappa^{\delta} \cdot \f{d \kappa^{\delta}}{d x}\bigg] dx= \al^2 (\al+ 3) \int  (\mathcal{J}^{\delta} \partial_x^6 y^{\delta})  \partial^2 \bigg[\f{(\mathcal{J}^{\delta} y^{\delta}_{x}) (\mathcal{J}^{\delta} y^{\delta}_{xx}) (\mathcal{J}^{\delta} y^{\delta}_{xxx})}{(1+ (\mathcal{J}^{\delta} y^{\delta}_x)^2)^3}\bigg] dx\\
&&-3 \al^2 (\al+ 3) \int  (\mathcal{J}^{\delta} \partial_x^6 y^{\delta})  \partial^2 \bigg[ \f{ (\mathcal{J}^{\delta} y^{\delta}_x)^2 \cdot (\mathcal{J}^{\delta} y^{\delta}_{xx})^3}{(1+ (\mathcal{J}^{\delta} y^{\delta}_x)^2)^4} \bigg] dx.
\end{eqnarray*}

Therefore the identity \eqref{es0} becomes the following:
\begin{eqnarray}\label{es1}
&&\hspace{5em}\f{1}{2} \partial_t \|\partial_x^4 y^{\delta}\|^2_{L^2}+  \al^2 (\al+ 3) \int    \f{(\mathcal{J}^{\delta} \partial_x^6 y^{\delta})^2}{(1+ (\mathcal{J}^{\delta} y^{\delta}_x)^2)^2} dx=\\ \nonumber
&&= -(\al- 1) \int (\mathcal{J}^{\delta} \partial_x^6 y^{\delta}) \cdot \partial_x^2 \bigg[\f{ (\mathcal{J}^{\delta} y^{\delta}_{xx})}{1+ (\mathcal{J}^{\delta} y^{\delta}_x)^2}\bigg] dx- \bigg(1+ \f{1}{2} \al^2\bigg)  \int (\mathcal{J}^{\delta} \partial_x^6 y^{\delta}) \cdot \partial_x^2 \bigg[\f{(\mathcal{J} y^{\delta}_{xx})^2}{(1+ (\mathcal{J}^{\delta} y^{\delta}_x)^2)^{\f{5}{2}}} \bigg] dx
\\ \nonumber
&&+ 4 \al^2 (\al+ 3) \int (\mathcal{J}^{\delta} \partial_x^6 y^{\delta}) \bigg[\f{(\mathcal{J}^{\delta} y^{\delta}_x) (\mathcal{J}^{\delta} y^{\delta}_{xx}) (\mathcal{J}^{\delta} \partial^5 y^{\delta})}{(1+ (\mathcal{J}^{\delta} y^{\delta}_x)^2)^3} \bigg] dx\\ \nonumber
&&+ 4 \al^2 (\al+ 3) \int (\mathcal{J}^{\delta} \partial_x^6 y^{\delta}) \partial \bigg[\f{ (\mathcal{J}^{\delta} y^{\delta}_x) (\mathcal{J}^{\delta} y^{\delta}_{xx}) (\mathcal{J}^{\delta} \partial^4 y^{\delta})}{(1+ (\mathcal{J}^{\delta} y^{\delta}_x)^2)^3} \bigg] dx
\\ \nonumber
&&- \int (\mathcal{J}^{\delta} \partial_x^6 y^{\delta}) \cdot \partial_x^2 \bigg[\sqrt{1+ (\mathcal{J}^{\delta} y^{\delta}_x)^2}\bigg] dx-  \bigg(2 \al+ 5 \al^2- \f{1}{3} \al^3\bigg)  \int (\mathcal{J}^{\delta} \partial_x^6 y^{\delta}) \cdot \partial_x^2 \bigg[\f{(\mathcal{J} y^{\delta}_{xx})^3}{(1+ (\mathcal{J}^{\delta} y^{\delta}_x)^2)^4} \bigg] dx\\ \nonumber
&& 3 \al^2 (\al+ 3) \int  (\mathcal{J}^{\delta} \partial_x^6 y^{\delta}) \cdot \partial_x^2 \bigg[ \f{  (\mathcal{J}^{\delta} y^{\delta}_{xx})^3}{(1+ (\mathcal{J}^{\delta} y^{\delta}_x)^2)^3} \bigg] dx+ 10 \al^2 (\al+ 3) \int  (\mathcal{J}^{\delta} \partial_x^6 y^{\delta}) \cdot \partial_x^2 \bigg[ \f{  (\mathcal{J}^{\delta} y^{\delta}_x) (\mathcal{J}^{\delta} y^{\delta}_{xx}) (\mathcal{J}^{\delta} y^{\delta}_{xxx})}{(1+ (\mathcal{J}^{\delta} y^{\delta}_x)^2)^3} \bigg] dx\\ \nonumber
&&- 18 \al^2 (\al+ 3) \int  (\mathcal{J}^{\delta} \partial_x^6 y^{\delta}) \cdot \partial_x^2 \bigg[ \f{ ( \mathcal{J}^{\delta} y^{\delta}_x)^2 (\mathcal{J}^{\delta} y^{\delta}_{xx})^3 }{(1+ (\mathcal{J}^{\delta} y^{\delta}_x)^2)^4} \bigg] dx= J_1+ J_2+ \cdots + J_9.
\end{eqnarray}
We claim that we can reduce the right hand side of this equality into a manageable form. In fact we will show that, for some $m> 2$ and $C_1$ small enough, 
\begin{eqnarray}\label{est4}
\bigg|J_1+ \cdots + I_9\bigg| \leq C_1 \int    \f{(\mathcal{J}^{\delta} \partial_x^6 y^{\delta})^2}{(1+ (\mathcal{J}^{\delta} y^{\delta}_x)^2)^2} dx+ C_2 \bigg(\| y^{\delta}\|^2_{L^2}+ \|\partial_x^4 y^{\delta}\|^2_{L^2}\bigg)+ C_3 \bigg(\| y^{\delta}\|^m_{L^2}+ \|\partial_x^4 y^{\delta}\|^m_{L^2}\bigg).
\end{eqnarray}
To prove this, we find bounds for each of the terms $J_1, \cdots, J_9$. 
Instead of demonstrating the full bound for every single integral, we focus on the most singular part of each of $J_1, \cdots, J_9,$ with these most singular parts being the terms with the highest derivatives when distributing spatial 
derivatives according to the product rule.
We will label collections of the less singular terms as $G(t)$, which stands for good terms.\\

We begin with $J_{1},$ estimating its most singular term by means of Young's inequality:
\begin{eqnarray*}
|J_1| &=&\bigg|(\al- 1) \int (\mathcal{J}^{\delta} \partial_x^6 y^{\delta}) \cdot \partial_x^2 \bigg[\f{ (\mathcal{J}^{\delta} y^{\delta}_{xx})}{1+ (\mathcal{J}^{\delta} y^{\delta}_x)^2}\bigg] dx\bigg| \leq \bigg|(\al- 1) \int (\mathcal{J}^{\delta} \partial_x^6 y^{\delta}) \cdot  \bigg[\f{ (\mathcal{J}^{\delta} \partial_x^4 y^{\delta})}{1+ (\mathcal{J}^{\delta} y^{\delta}_x)^2}\bigg] dx\bigg|+ G(t)\\
& \leq& \f{1}{100} \bigg\|\f{\mathcal{J}^{\delta} \partial_x^6 y^{\delta}}{1+ (\mathcal{J}^{\delta} y^{\delta}_x)^2} \bigg\|^2_{L^2}+ C \| \mathcal{J}^{\delta} \partial_x^4 y^{\delta}\|^2_{L^2}+ G(t).
\end{eqnarray*}

We proceed similarly for the most singular term in $J_{2}:$
\begin{eqnarray*}
	|J_2| &=& \bigg|\bigg(1+ \f{1}{2} \al^2\bigg)  \int (\mathcal{J}^{\delta} \partial_x^6 y^{\delta}) \cdot \partial_x^2 \bigg[\f{(\mathcal{J} y^{\delta}_{xx})^2}{(1+ (\mathcal{J}^{\delta} y^{\delta}_x)^2)^{\f{5}{2}}} \bigg] dx\bigg| \leq C \bigg| \int  \f{(\mathcal{J}^{\delta} \partial_x^6 y^{\delta})}{1+ (\mathcal{J}^{\delta} y^{\delta}_x)} \cdot  \bigg[\f{(\mathcal{J}^{\delta} \partial_x^2 y^{\delta}) (\mathcal{J}^{\delta} \partial_x^4 y^{\delta})}{1+ (\mathcal{J}^{\delta} y^{\delta}_x)^{\f{3}{2}}}\bigg] dx\bigg|+ G(t)\\
	& \leq& \f{1}{100} \bigg\|\f{\mathcal{J}^{\delta} \partial_x^6 y^{\delta}}{1+ (\mathcal{J}^{\delta} y^{\delta}_x)^2} \bigg\|^2_{L^2}+ C \left\| \mathcal{J}^{\delta} \partial_x^2 y^{\delta} \right\|_{L^{\infty}}^2 
	\left\| \mathcal{J}^{\delta} \partial_x^4 y^{\delta}\right\|^2_{L^2}+ G(t).
\end{eqnarray*}
We now use the Sobelev as well as Gagliardo-Nirenberg inequalities to control $\| \mathcal{J}^{\delta} \partial_x^2 y^{\delta} \|_{L^{\infty}}:$
\begin{eqnarray*}
\|  \partial_x^2 y^{\delta} \|_{L^{\infty}} \leq  \| |\nabla|^{\f{1}{2}}\partial_x^2 y^{\delta} \|_{L^2}  \leq  \| y^{\delta} \|^{\f{3}{8}}_{L^2}  \|\partial_x^4 y^{\delta} \|^{\f{5}{8}}_{L^2}.
\end{eqnarray*}
This then implies
\begin{eqnarray*}
	|J_2| &\leq&  \f{1}{100} \bigg\|\f{\mathcal{J}^{\delta} \partial_x^6 y^{\delta}}{1+ (\mathcal{J}^{\delta} y^{\delta}_x)^2} \bigg\|^2_{L^2}+ C \bigg(\| y^{\delta} \|^{\f{3}{8}}_{L^2}  \|\partial_x^4 y^{\delta} \|^{\f{5}{8}}_{L^2}\bigg)^2 \| \partial_x^4 y^{\delta}\|^2_{L^2}+ G(t)\\
	&\leq& \f{1}{100} \left\|\f{\mathcal{J}^{\delta} \partial_x^6 y^{\delta}}{1+ (\mathcal{J}^{\delta} y^{\delta}_x)^2} 
	\right\|^2_{L^2}+ C \left(\left\| y^{\delta} \right\|^4_{L^2} + \left\|\partial_x^4 y^{\delta} \right\|^4_{L^2}\right)+ G(t).
\end{eqnarray*}

We turn our attention to estimating $J_{3};$ to begin, we have
\begin{eqnarray*}
	|J_3| & = & \left|4 \al^2 (\al+ 3) \int (\mathcal{J}^{\delta} \partial_x^6 y^{\delta}) \left[\f{(\mathcal{J}^{\delta} y^{\delta}_x) (\mathcal{J}^{\delta} y^{\delta}_{xx}) (\mathcal{J}^{\delta} \partial^5 y^{\delta})}{(1+ (\mathcal{J}^{\delta} y^{\delta}_x)^2)^3} \right] \ dx\right|\\
	& \leq& C \left| \int  \f{(\mathcal{J}^{\delta} \partial_x^6 y^{\delta})}{1+ (\mathcal{J}^{\delta} y^{\delta}_x)^2} \cdot  \left[\f{ (\mathcal{J}^{\delta}  y^{\delta}_x) (\mathcal{J}^{\delta} y_{xx}^{\delta}) (\mathcal{J}^{\delta} \partial_x^5 y^{\delta})}{(1+ (\mathcal{J}^{\delta} y^{\delta}_x)^2)^2}\right] \ dx\right|+ G(t)\\
	& \leq& \f{1}{100} \left\|\f{\mathcal{J}^{\delta} \partial_x^6 y^{\delta}}{1+ (\mathcal{J}^{\delta} y^{\delta}_x)^2} 
	\right\|^2_{L^2}+ C \left\| \mathcal{J}^{\delta} \partial_x^2 y^{\delta} \right\|_{L^{\infty}}^2 
	\left\| \mathcal{J}^{\delta} \partial_x^5 y^{\delta}\right\|^2_{L^2}+ G(t).
\end{eqnarray*}
Here we have used the fact that $\bigg|\f{ (\mathcal{J}^{\delta}  y^{\delta}_x) }{(1+ (\mathcal{J}^{\delta} y^{\delta}_x)^2)^2}\bigg| \leq 1$. 

We turn our attention to bounding $\| \mathcal{J}^{\delta} \partial_x^2 y^{\delta} \|_{L^{\infty}}^2$ and $\| \mathcal{J}^{\delta} \partial_x^5 y^{\delta}\|^2_{L^2}$ as follows. We use the Sobelev inequality as well as the
Gagliardo-Nirenberg inequality, finding  
\begin{eqnarray}\label{es2}
\| \mathcal{J}^{\delta} \partial_x^2 y^{\delta} \|_{L^{\infty}} \leq \| \mathcal{J}^{\delta} |\nabla|^{\f{1}{2}}\partial_x^2 y^{\delta} \|_{L^2} \leq \| y^{\delta} \|^{\f{3}{8}}_{L^2} \|  \partial_x^4 y^{\delta} \|^{\f{5}{8}}_{L^2}.
\end{eqnarray}
Moreover,
\begin{eqnarray*}
 \| \mathcal{J}^{\delta} \partial_x^5 y^{\delta}\|_{L^2} \leq \| \mathcal{J}^{\delta} \partial^6_x y^{\delta} \|^{\f{1}{2}}_{L^2} \|  \partial_x^4 y^{\delta} \|^{\f{1}{2}}_{L^2} \leq C  \| \f{\mathcal{J}^{\delta} \partial_x^6 y^{\delta}}{1+ (\mathcal{J}^{\delta} y^{\delta}_x)^2} \|^{\f{1}{2}}_{L^2} \|  \partial_x^4 y^{\delta} \|^{\f{1}{2}}_{L^2} \|1+ (\mathcal{J}^{\delta} y^{\delta}_x)^2 \|^{\f{1}{2}}_{L^{\infty}},
\end{eqnarray*}
and also,
\begin{eqnarray*}
\|1+ (\mathcal{J}^{\delta} y^{\delta}_x)^2 \|_{L^{\infty}} \leq 1+  \|\mathcal{J}^{\delta} y^{\delta}_x \|^2_{L^{\infty}} \leq 1+  \|\mathcal{J}^{\delta} |\nabla|^{\f{1}{2}} y^{\delta}_x \|^2_{L^2} \leq 1+  \bigg(\|\mathcal{J}^{\delta} y^{\delta} \|^{\f{5}{8}}_{L^2}  \|\mathcal{J}^{\delta} \partial^4_x y^{\delta} \|^{\f{3}{8}}_{L^2}\bigg)^2.
\end{eqnarray*}

We may thus conclude our bound for $J_{3}:$
\begin{eqnarray*}
|J_3| &\leq& \f{1}{100} \bigg\|\f{\mathcal{J}^{\delta} \partial_x^6 y^{\delta}}{1+ (\mathcal{J}^{\delta} y^{\delta}_x)^2} \bigg\|^2_{L^2}+ C \| \f{\mathcal{J}^{\delta} \partial_x^6 y^{\delta}}{1+ (\mathcal{J}^{\delta} y^{\delta}_x)^2} \|_{L^2} \|  y^{\delta} \|^2_{L^2} \|  \partial_x^4 y^{\delta} \|^3_{L^2}+ C +  G(t)\\
&\leq &\f{2}{100} \bigg\|\f{\mathcal{J}^{\delta} \partial_x^6 y^{\delta}}{1+ (\mathcal{J}^{\delta} y^{\delta}_x)^2} \bigg\|^2_{L^2}+ C  \|  y^{\delta} \|^4_{L^2} \|  \partial_x^4 y^{\delta} \|^6_{L^2}+ C+ G(t)\\
&\leq &\f{2}{100} \bigg\|\f{\mathcal{J}^{\delta} \partial_x^6 y^{\delta}}{1+ (\mathcal{J}^{\delta} y^{\delta}_x)^2} \bigg\|^2_{L^2}+ C  \bigg(\|  y^{\delta} \|^{10}_{L^2}+ \|  \partial_x^4 y^{\delta} \|^{10}_{L^2}\bigg)+ G(t).
\end{eqnarray*}
Note that above we used the assumption that $\|  y^{\delta} \|^2_{L^2}+ \|  \partial_x^4 y^{\delta} \|^2_{L^2} \geq 1$ (otherwise there would be nothing to prove), and consequently $C < \bigg(\|  y^{\delta} \|^{10}_{L^2}+ \|  \partial_x^4 y^{\delta} \|^{10}_{L^2}\bigg)$.

We estimate $J_{4}$ similarly to how we estimated $J_{3}:$
\begin{eqnarray*}
	|J_4| & = & \bigg|4 \al^2 (\al+ 3) \int (\mathcal{J}^{\delta} \partial_x^6 y^{\delta}) \partial \bigg[\f{ (\mathcal{J}^{\delta} y^{\delta}_x) (\mathcal{J}^{\delta} y^{\delta}_{xx}) (\mathcal{J}^{\delta} \partial^4 y^{\delta})}{(1+ (\mathcal{J}^{\delta} y^{\delta}_x)^2)^3} \bigg] dx\bigg|\\
	& \leq& C \bigg| \int  \f{(\mathcal{J}^{\delta} \partial_x^6 y^{\delta})}{1+ (\mathcal{J}^{\delta} y^{\delta}_x)^2} \cdot  \bigg[\f{ (\mathcal{J}^{\delta}  y^{\delta}_x) (\mathcal{J}^{\delta} y_{xx}^{\delta}) (\mathcal{J}^{\delta} \partial_x^5 y^{\delta})}{(1+ (\mathcal{J}^{\delta} y^{\delta}_x)^2)^2}\bigg] dx\bigg|+ G(t)\\
	& \leq& \f{2}{100} \bigg\|\f{\mathcal{J}^{\delta} \partial_x^6 y^{\delta}}{1+ (\mathcal{J}^{\delta} y^{\delta}_x)^2} \bigg\|^2_{L^2}+ C  \bigg(\|  y^{\delta} \|^{10}_{L^2}+ \|  \partial_x^4 y^{\delta} \|^{10}_{L^2}\bigg)+ G(t).
\end{eqnarray*}

We next consider $J_{5},$ beginning as follows:
\begin{eqnarray*}
	|J_5| &=& \bigg|\int (\mathcal{J}^{\delta} \partial_x^6 y^{\delta}) \cdot \partial_x^2 \bigg[\f{1+ ( \mathcal{J}^{\delta} y^{\delta}_x)^2}{\sqrt{1+ (\mathcal{J}^{\delta}  y^{\delta}_x)^2} }\bigg] dx\bigg|\\
	& \leq& C \bigg| \int  (\mathcal{J}^{\delta} \partial_x^6 y^{\delta}) \cdot  \bigg[\f{ (\mathcal{J}^{\delta}  y^{\delta}_x)  (\mathcal{J}^{\delta} \partial_x^3 y^{\delta})}{\sqrt{1+ ( \mathcal{J}^{\delta} y^{\delta}_x)^2}}\bigg] dx\bigg|+ G(t)\\
	& \leq& \f{1}{100} \bigg\|\f{\mathcal{J}^{\delta} \partial_x^6 y^{\delta}}{1+ (\mathcal{J}^{\delta} y^{\delta}_x)^2} \bigg\|^2_{L^2}+ C \| (\mathcal{J}^{\delta}  y^{\delta}_x)  (\mathcal{J}^{\delta} \partial_x^3 y^{\delta}) \sqrt{1+ ( \mathcal{J}^{\delta} y^{\delta}_x)^2}\|^2_{L^2}+ G(t)\\
	& \leq& \f{1}{100} \bigg\|\f{\mathcal{J}^{\delta} \partial_x^6 y^{\delta}}{1+ (\mathcal{J}^{\delta} y^{\delta}_x)^2} \bigg\|^2_{L^2}+ C \| (\mathcal{J}^{\delta}  y^{\delta}_x)\|^2_{L^{\infty}}  \|(\mathcal{J}^{\delta} \partial_x^3 y^{\delta})\|^2_{L^2} \|\sqrt{1+ ( \mathcal{J}^{\delta} y^{\delta}_x)^2}\|^2_{L^{\infty}}+ G(t).
\end{eqnarray*}
By the Sobelev and Gagliardo-Nirenberg inequalities, we have
\begin{eqnarray*}
\left\| (\mathcal{J}^{\delta}  y^{\delta}_x)\right\|_{L^{\infty}} \leq 
\left\| |\nabla|^{\f{1}{2}} y^{\delta}_x\right\|_{L^2} \leq \left\|y^{\delta}\right\|^{\f{5}{8}}_{L^2} 
\left\|\partial^4_x y^{\delta} \right\|^{\f{3}{8}}_{L^2},
\end{eqnarray*} 
as well as
\begin{eqnarray*}
	\left\|\sqrt{1+ ( \mathcal{J}^{\delta} y^{\delta}_x)^2}\right\|_{L^{\infty}}^2 \leq C \bigg(1+  \|\mathcal{J}^{\delta}  y^{\delta}_x \|_{L^{\infty}}\bigg)^2 \leq C \bigg( 1+ \| |\nabla|^{\f{1}{2}} y^{\delta}_x\|_{L^2}\bigg)^2 \leq C \bigg( 1+  \|y^{\delta}\|^{\f{5}{8}}_{L^2} \|\partial^4_x y^{\delta} \|^{\f{3}{8}}_{L^2}\bigg)^2.
\end{eqnarray*} 
Moreover,
\begin{eqnarray*}
	\| (\mathcal{J}^{\delta} \partial_x^3 y^{\delta})\|_{L^2}  \leq \|y^{\delta}\|^{\f{1}{4}}_{L^2} \|\partial^4_x y^{\delta} \|^{\f{3}{4}}_{L^2}.
\end{eqnarray*} 
We may then conclude our bound for $J_{5}$ as
\begin{eqnarray*}
	|J_5|	& \leq& \f{1}{100} \bigg\|\f{\mathcal{J}^{\delta} \partial_x^6 y^{\delta}}{1+ (\mathcal{J}^{\delta} y^{\delta}_x)^2} \bigg\|^2_{L^2}+ C \|y^{\delta}\|^3_{L^2} \|\partial^4_x y^{\delta} \|^3_{L^2}+ 1+ G(t)\\
	&\leq& \f{1}{100} \bigg\|\f{\mathcal{J}^{\delta} \partial_x^6 y^{\delta}}{1+ (\mathcal{J}^{\delta} y^{\delta}_x)^2} \bigg\|^2_{L^2}+ C \|y^{\delta}\|^4_{L^2} + \|\partial^4_x y^{\delta} \|^4_{L^2}+ C \|y^{\delta}\|^6_{L^2} + \|\partial^4_x y^{\delta} \|^6_{L^2}+ G(t).
\end{eqnarray*}

We begin the estimate for $J_{6}$ similarly to how we estimated $J_{3}$ above:
\begin{eqnarray*}
 |J_6| &  = & \left| \bigg(2 \al+ 5 \al^2- \f{1}{3} \al^3\bigg)  \int (\mathcal{J}^{\delta} \partial_x^6 y^{\delta}) \cdot \partial_x^2 \bigg[\f{(\mathcal{J} y^{\delta}_{xx})^3}{(1+ (\mathcal{J}^{\delta} y^{\delta}_x)^2)^4} \bigg] dx \right|\\
 & \leq& C \bigg|   \int \f{(\mathcal{J}^{\delta} \partial_x^6 y^{\delta})}{(1+ (\mathcal{J}^{\delta} y^{\delta}_x)^2)} \cdot  \bigg[\f{(\mathcal{J}^{\delta} y^{\delta}_{xx})^2 (\mathcal{J}^{\delta} \partial^4 y^{\delta})}{(1+ (\mathcal{J}^{\delta} y^{\delta}_x)^2)^3} \bigg] dx \bigg|+ G(t)\\
 & \leq& \f{2}{100} \bigg\|\f{\mathcal{J}^{\delta} \partial_x^6 y^{\delta}}{1+ (\mathcal{J}^{\delta} y^{\delta}_x)^2} \bigg\|^2_{L^2}+ C  \bigg\|(\mathcal{J}^{\delta} y^{\delta}_{xx})^2 (\mathcal{J}^{\delta} \partial^4 y^{\delta}) \bigg\|_{L^2}+ G(t).
\end{eqnarray*}
We then make use of relation \eqref{es2}, finding
\begin{eqnarray*}
\bigg\|(\mathcal{J}^{\delta} y^{\delta}_{xx})^2 (\mathcal{J}^{\delta} \partial^4 y^{\delta}) \bigg\|_{L^2} \leq \| y^{\delta}\|_{L^2}^{\f{3}{2}} \|\mathcal{J}^{\delta} \partial^4 y^{\delta}\|_{L^2}^{\f{9}{2}}  \leq C \bigg(\| y^{\delta}\|_{L^2}^6+ \|\mathcal{J}^{\delta} \partial^4 y^{\delta}\|_{L^2}^6\bigg).
\end{eqnarray*}
Therefore, we have the conclusion
\begin{eqnarray*}
	|J_6| & \leq&  \f{2}{100} \bigg\|\f{\mathcal{J}^{\delta} \partial_x^6 y^{\delta}}{1+ (\mathcal{J}^{\delta} y^{\delta}_x)^2} \bigg\|^2_{L^2}+ C \bigg(\| y^{\delta}\|_{L^2}^6+ \|\mathcal{J}^{\delta} \partial^4 y^{\delta}\|_{L^2}^6\bigg)+ G(t).
\end{eqnarray*}

We estimate $J_{7}$ as follows:
\begin{eqnarray}\nonumber
	&&|J_7|  = \left| 3 \al^2 (\al+ 3) \int  \f{(\mathcal{J}^{\delta} \partial_x^6 y^{\delta})}{1+ (\mathcal{J}^{\delta} y^{\delta}_x)^2} \cdot \partial_x^2 \left[ \f{  (\mathcal{J}^{\delta} y^{\delta}_{xx})^3}{(1+ (\mathcal{J}^{\delta} y^{\delta}_x)^2)^2} \right]\ dx \right|\\ \nonumber
	&& \leq C  \left| \int  \f{(\mathcal{J}^{\delta} \partial_x^6 y^{\delta})}{1+ (\mathcal{J}^{\delta} y^{\delta}_x)^2} \cdot  \left[ \f{  (\mathcal{J}^{\delta} y^{\delta}_{xx})^2 (\mathcal{J}^{\delta} \partial_x^4 y^{\delta}) }{(1+ (\mathcal{J}^{\delta} y^{\delta}_x)^2)^2} \right]\ dx \right|+ G(t)\\  \label{es3}
	&& \leq\f{1}{100} \left\|\f{\mathcal{J}^{\delta} \partial_x^6 y^{\delta}}{1+ (\mathcal{J}^{\delta} y^{\delta}_x)^2} \right\|^2_{L^2}+ C  \left\| \f{  (\mathcal{J}^{\delta} y^{\delta}_{xx})^2 (\mathcal{J}^{\delta} \partial_x^4 y^{\delta}) }{(1+ (\mathcal{J}^{\delta} y^{\delta}_x)^2)^2}\right\|^2_{L^2}+ G(t)\\ \nonumber
	& &\leq  \f{1}{100} \left\|\f{\mathcal{J}^{\delta} \partial_x^6 y^{\delta}}{1+ (\mathcal{J}^{\delta} y^{\delta}_x)^2} \right\|^2_{L^2}+ C  \left(\|   \mathcal{J}^{\delta} y^{\delta}_{xx}\|^2_{L^{\infty}} \|\mathcal{J}^{\delta} \partial_x^4 y^{\delta}\|_{L^2} \right)^2+ G(t)\\\nonumber
	&&\leq  \f{1}{100} \left\|\f{\mathcal{J}^{\delta} \partial_x^6 y^{\delta}}{1+ (\mathcal{J}^{\delta} y^{\delta}_x)^2} \right\|^2_{L^2}+ C  \left(\| y^{\delta} \|^{\f{3}{8}}_{L^2} \|  \partial_x^4 y^{\delta} \|^{\f{5}{8}}_{L^2} \right)^4 \|\mathcal{J}^{\delta} \partial_x^4 y^{\delta}\|^2_{L^2}+ G(t).
	\end{eqnarray}
Our conclusion for $J_{7}$ is then
\begin{eqnarray*}
|J_7| \leq   \f{1}{100} \left\|\f{\mathcal{J}^{\delta} \partial_x^6 y^{\delta}}{1+ (\mathcal{J}^{\delta} y^{\delta}_x)^2} \right\|^2_{L^2}+ C  \left(\| y^{\delta} \|^6_{L^2}+ \|\mathcal{J}^{\delta} \partial_x^4 y^{\delta}\|^6_{L^2}\right)+ G(t).
\end{eqnarray*}

For $J_{8},$ we begin with the following estimate:
\begin{eqnarray*}
|J_8| &=&10 \left| \al^2 (\al+ 3) \int  (\mathcal{J}^{\delta} \partial_x^6 y^{\delta}) \cdot \partial_x^2 \left[ \f{  (\mathcal{J}^{\delta} y^{\delta}_x) (\mathcal{J}^{\delta} y^{\delta}_{xx}) (\mathcal{J}^{\delta} y^{\delta}_{xxx})}{(1+ (\mathcal{J}^{\delta} y^{\delta}_x)^2)^3} \right]\ dx \right|\\
& \leq& C \left| \int  (\mathcal{J}^{\delta} \partial_x^6 y^{\delta}) \cdot \left[ \f{  (\mathcal{J}^{\delta} y^{\delta}_x) (\mathcal{J}^{\delta} y^{\delta}_{xx}) (\mathcal{J}^{\delta} \partial^5 y^{\delta})}{(1+ (\mathcal{J}^{\delta} y^{\delta}_x)^2)^3} \right]\ dx \right|+ G(t).
\end{eqnarray*}
The integral on the right-hand side is similar to $J_3$, and we handle it in the same way.

This brings us to the final term to estimate, $J_{9};$ for this, we have the following: 
\begin{eqnarray*}
|J_9| &=& \left|18 \al^2 (\al+ 3) \int  (\mathcal{J}^{\delta} \partial_x^6 y^{\delta}) \cdot \partial_x^2 \left[ \f{ ( \mathcal{J}^{\delta} y^{\delta}_x)^2 (\mathcal{J}^{\delta} y^{\delta}_{xx})^3 }{(1+ (\mathcal{J}^{\delta} y^{\delta}_x)^2)^4} \right] dx\right|\\
& \leq& C \left| \int  \f{(\mathcal{J}^{\delta} \partial_x^6 y^{\delta})}{1+ (\mathcal{J}^{\delta} y^{\delta}_x)^2} \cdot \left[ \f{  (\mathcal{J}^{\delta} y^{\delta}_x)^2 (\mathcal{J}^{\delta} y^{\delta}_{xx})^2 (\mathcal{J}^{\delta} \partial^4 y^{\delta})}{(1+ (\mathcal{J}^{\delta} y^{\delta}_x)^2)^3} \right] dx \right|+ G(t)\\
&\leq&\f{1}{100} \left\|\f{\mathcal{J}^{\delta} \partial_x^6 y^{\delta}}{1+ (\mathcal{J}^{\delta} y^{\delta}_x)^2} \right\|^2_{L^2}+ C  \left\|   (\mathcal{J}^{\delta} y^{\delta}_{xx})^2 (\mathcal{J}^{\delta} \partial_x^4 y^{\delta}) \right\|^2_{L^2}+ G(t).
\end{eqnarray*}
Here we have used the fact that $\left|\f{(\mathcal{J}^{\delta} y^{\delta}_x)^2}{(1+ (\mathcal{J}^{\delta} y^{\delta}_x)^2)^3}\right| \leq 1$. The last inequality is similar to \eqref{es3}, and we proceed in the same way to get
\begin{eqnarray*}
	|J_9| &\leq& \f{1}{100} \bigg\|\f{\mathcal{J}^{\delta} \partial_x^6 y^{\delta}}{1+ (\mathcal{J}^{\delta} y^{\delta}_x)^2} \bigg\|^2_{L^2}+ C  \bigg(\| y^{\delta} \|^6_{L^2}+ \|\mathcal{J}^{\delta} \partial_x^4 y^{\delta}\|^6_{L^2}\bigg)+ G(t).
\end{eqnarray*}

Putting the above together leads to the relation \eqref{est4}, which we had been aiming to prove. 
We now may add \eqref{en30} and \eqref{est4}, finding
\begin{eqnarray*}
\f{1}{2} \partial_t \left(\| y^{\delta}\|^2_{L^2}+ \|\partial_x^4 y^{\delta}\|^2_{L^2} \right)+  C \int    \f{(\mathcal{J}^{\delta} \partial_x^2 y^{\delta})^2+ (\mathcal{J}^{\delta} \partial_x^6 y^{\delta})^2}{(1+ (\mathcal{J}^{\delta} y^{\delta}_x)^2)^2} dx \leq C_0 \left(\| y^{\delta}\|^2_{L^2}+ \|\partial_x^4 y^{\delta}\|^2_{L^2}\right)+ C_1 \left(\| y^{\delta}\|^m_{L^2}+ \|\partial_x^4 y^{\delta}\|^m_{L^2}\right).
\end{eqnarray*}
This inequality  clearly implies a uniform bound (uniform with respect to $\delta$) for the function $y^{\delta}(x, t)$ in the space $H^4$ until a time $T= T(\al, \|y_0\|_{H^4})$. We will study the size of this time interval $[0, T]$ in a bit more
detail, and to this end we define $I(t)= \|y^{\delta}\|^2_{L^2}+ \|\partial_x^4 y^{\delta}\|^2_{L^2} $.  We  also 
fix  $C_0$  with $\ga= C_0$.

We may then say
\begin{eqnarray*}
	I(t) \leq e^{\ga t} I(0)+ \int_0^t e^{ \ga (t- s)} I^{\f{m}{2}}(s) ds.
\end{eqnarray*}
We use Lemma \ref{gr} to see that $I(t)$ remains bounded as long as $t \in [0, \be_{\f{m}{2}}],$ where 
\begin{eqnarray}
\be_{\f{m}{2} }= \sup \left\{t: \left(\f{m}{2}- 1\right) \int_0^t e^{- C\ga s} e^{\ga s} \bigg[\bigg(I(0) \bigg)^{(\f{m}{2}- 1)} e^{(\f{m}{2}- 1) \ga s}\bigg] ds < 1 \right\}.
\end{eqnarray}
A simple calculation then shows that we have guaranteed existence of our solutions over the interval
\begin{eqnarray}\label{tt00}
0 < t< \f{\ln\bigg(1+ \f{\ga}{\|y_0\|^{m- 2}_{H^4}}\bigg)}{ \ga}.
\end{eqnarray}
Clearly, this bound for the time of existence depends on the initial values; that is if the  value $I(0)= \|y_0\|^2_{L^2}+ \|\partial_x^4 y_0\|^2_{L^2}$ stays small, the time interval is large.  Note that we will take this $m> 2$ and $\ga$ to be fixed throughout the sequel. 
\end{proof} 
Lemma \ref{lem00} provides a uniform bound in $H^4$ for the solutions of the approximate equations \eqref{yt0}. Although this is a good and useful estimate, as we are aiming to show the existence of classical solutions, we need
a little more. In what follows, we will pass to the limit of solutions of the approximate equations \eqref{yt0} to find
solutions of the original equation \eqref{yt}. In order to do this, we need to have at least the continuity of the function $F(y^{\de})$, where $F(y^{\de})$ denotes
\begin{eqnarray}\label{FF} 
y^{\de}_t= F(y^{\de}),
\end{eqnarray}
with $y^{\de}$ determined from \eqref{yt0}.
To guarantee continuity of this function, one approach is to prove an  $H^5$ uniform  bound. This clearly means $\partial^4_x y$ is continuous and hence the function $F(y^{\de})$ is continuous as well. The following lemma concerns the appropriate bound. 
\begin{lem}\label{lem01}
	Let $y^{\delta}$ be the solution of  \eqref{yt}. Then there exists $T= T(\al)$ and $C= C(y_0, \al)$, independent of $\delta$, so that for any  $0 < t< \f{\ln\bigg(1+ \f{\ga}{\|y_0\|^{m- 2}_{H^4}}\bigg)}{ \ga}$, 
	\begin{eqnarray}
	\sup_{0< t< T} \|y^{\delta}\|^2_{H^5}+   \int  \f{ (\partial_x^7 \mathcal{J}^{\delta} y^{\delta})^2}{(1+ (\mathcal{J}^{\delta} y^{\delta}_x)^2)^2}   dx \leq C.
	\end{eqnarray}
\end{lem}
\begin{proof}
	We take five spatial derivatives of equation \eqref{yt0}. Then its inner product with the function $\partial_x^5 y^{\de}$ leads to the following identity:
	\begin{eqnarray}\label{en01}
	&& \f{1}{2} \partial_t \|\partial^5_x y^{\delta}\|^2_{L^2}+ (\al- 1) \int ( \mathcal{J}^{\delta} \partial^5_x y^{\delta}) \cdot  \partial^5_x \bigg[\f{ (\mathcal{J}^{\delta } y^{\delta}_{xx})}{1+ (\mathcal{J}^{\delta} y^{\delta}_x)^2}\bigg] dx+\\ \nonumber
	&+& \al^2(\al+ 3) \int  (\mathcal{J}^{\delta} \partial^5_x y^{\delta}) \cdot   \partial^5_x \bigg[\f{1}{\sqrt{1+ (\mathcal{J}^{\delta} y^{\delta}_x)^2}} \cdot\f{d^2 \kappa^{\delta}}{d x^2}\bigg] dx+ \int (\mathcal{J}^{\delta} \partial^5_x y^{\delta}) \cdot  \partial^5_x  \bigg[\sqrt{1+ (\mathcal{J}^{\delta} \mathcal{J}^{\delta} y^{\delta}_x)^2}\bigg] dx+\\ \nonumber
		&&+ \bigg(1+ \f{1}{2} \al^2\bigg)  \int (\mathcal{J}^{\delta} \partial_x^5 y^{\delta}) \cdot \partial_x^5 \bigg[\f{(\mathcal{J} y^{\delta}_{xx})^2}{(1+ (\mathcal{J}^{\delta} y^{\delta}_x)^2)^{\f{5}{2}}} \bigg] dx+\\ \nonumber
		&&+  \bigg(2 \al+ 5 \al^2- \f{1}{3} \al^3\bigg)  \int (\mathcal{J}^{\delta} \partial_x^5 y^{\delta}) \cdot \partial_x^5 \bigg[\f{(\mathcal{J} y^{\delta}_{xx})^3}{(1+ (\mathcal{J}^{\delta} y^{\delta}_x)^2)^4} \bigg] dx\\ \nonumber
	&=& \al^2(\al+ 3) \int (\mathcal{J}^{\delta} \partial^5_x y^{\delta}) \cdot  \partial^5_x \bigg[(\mathcal{J}^{\delta} \mathcal{J}^{\delta} y^{\delta}_x) \cdot \kappa^{\delta} \cdot \f{d \kappa^{\delta}}{d x}\bigg] dx.
	\end{eqnarray}
	For a more convenient form of this identity, we simplify some of these terms.  To begin, we have
	\begin{eqnarray*}
		&&\al^2(\al+ 3) \int  (\mathcal{J}^{\delta} \partial^5_x y^{\delta}) \cdot  \partial^5_x \bigg[\f{1}{\sqrt{1+ (\mathcal{J}^{\delta} y^{\delta}_x)^2}} \cdot\f{d^2 \kappa^{\delta}}{d x^2}\bigg] dx= \al^2(\al+ 3) \int  (\mathcal{J}^{\delta} \partial^5_x y^{\delta}) \cdot  \partial^5_x \bigg[	 \f{\mathcal{J}^{\delta} y^{\delta}_{xxxx}}{(1+ (\mathcal{J}^{\delta} y^{\delta}_x)^2)^2}\bigg] dx\\
		&&-  3 \al^2(\al+ 3) \int  (\mathcal{J}^{\delta} \partial^5_x y^{\delta}) \cdot  \partial^5_x \bigg[\f{  (\mathcal{J}^{\delta} y^{\delta}_{xx})^3}{(1+ (\mathcal{J}^{\delta} y^{\delta}_x)^2)^3} \bigg] dx+ 9 \al^2(\al+ 3) \int  (\mathcal{J}^{\delta} \partial^5_x y^{\delta}) \cdot  \partial^5_x \bigg[\f{   (\mathcal{J}^{\delta} y^{\delta}_x) (\mathcal{J}^{\delta} y^{\delta}_{xx}) (\mathcal{J}^{\delta} y^{\delta}_{xxx})}{(1+ (\mathcal{J}^{\delta} y^{\delta}_x)^2)^3} \bigg] dx\\
		&&+ 15 \al^2(\al+ 3) \int  (\mathcal{J}^{\delta} \partial^5_x y^{\delta}) \cdot  \partial^5_x \bigg[\f{ ( \mathcal{J}^{\delta} y^{\delta}_x)^2 (\mathcal{J}^{\delta} y^{\delta}_{xx})^3 }{(1+ (\mathcal{J}^{\delta} y^{\delta}_x)^2)^4}\bigg] dx.
	\end{eqnarray*}
	The first term in the right hand side of this relation is
	\begin{eqnarray*}
		&&\al^2(\al+ 3) \int  (\mathcal{J}^{\delta} \partial^5_x y^{\delta}) \cdot  \partial^5_x \bigg[	 \f{\mathcal{J}^{\delta} y^{\delta}_{xxxx}}{(1+ (\mathcal{J}^{\delta} y^{\delta}_x)^2)^2}\bigg] dx= \al^2(\al+ 3) \int   	 \f{(\mathcal{J}^{\delta} \partial^7_x y^{\delta})^2}{(1+ (\mathcal{J}^{\delta} y^{\delta}_x)^2)^2} dx-\\
		&&- 4 \al^2(\al+ 3) \int ( \mathcal{J}^{\delta} \partial_x^7 y^\de) \cdot  \f{( \mathcal{J}^{\delta} \partial_x^6 y^\de) ( \mathcal{J}^{\delta} y^\de_{x}) (y^\de_{xx})}{(1+ (\mathcal{J}^{\delta} y^{\delta}_x)^2)^3} dx\\
		&&- 4 \al^2(\al+ 3) \int ( \mathcal{J}^{\delta} \partial_x^7 y^\de) \cdot  \partial_x^2 \bigg[\f{( \mathcal{J}^{\delta} \partial_x^4 y^\de) ( \mathcal{J}^{\delta} y^\de_{x}) (y^\de_{xx})}{(1+ (\mathcal{J}^{\delta} y^{\delta}_x)^2)^3} \bigg] dx.
	\end{eqnarray*}
	We next rewrite another term appearing in \eqref{en01}:
	\begin{eqnarray*}
		\int (\mathcal{J}^{\delta} \partial^5_x y^{\delta}) \cdot  \partial^5_x  \bigg[\sqrt{1+ (\mathcal{J}^{\delta} \mathcal{J}^{\delta} y^{\delta}_x)^2}\bigg] dx= \int (\mathcal{J}^{\delta} \partial^7_x y^{\delta}) \cdot  \partial^3_x  \bigg[\f{1+ (\mathcal{J}^{\delta} \mathcal{J}^{\delta} y^{\delta}_x)^2}{\sqrt{1+ (\mathcal{J}^{\delta} \mathcal{J}^{\delta} y^{\delta}_x)^2} }\bigg] dx.
	\end{eqnarray*}
	Finally, another term in \eqref{en01} can be written as follows:
	\begin{eqnarray*}
		\al^2(\al+ 3) \int (\mathcal{J}^{\delta} \partial^5_x y^{\delta}) \cdot  \partial^5_x \bigg[(\mathcal{J}^{\delta} \mathcal{J}^{\delta} y^{\delta}_x) \cdot \kappa^{\delta} \cdot \f{d \kappa^{\delta}}{d x}\bigg] dx&=&  4 \int (\mathcal{J}^{\delta} \partial^7_x y^{\delta}) \cdot  \partial^3_x \bigg[  \f{  (\mathcal{J}^{\delta} \mathcal{J}^{\delta} y^{\delta}_x) (\mathcal{J}^{\delta} \mathcal{J}^{\delta} y^{\delta}_{xx}) (\mathcal{J}^{\delta} y^{\delta}_{xxx}) }{(1+ (\mathcal{J}^{\delta} y^{\delta}_x)^2)^3}\bigg] dx\\
		&-& 4 \al^2(\al+ 3) \int (\mathcal{J}^{\delta} \partial^7_x y^{\delta}) \cdot  \partial^3_x \bigg[    \f{ (\mathcal{J}^{\delta} y^{\delta}_x)^2 \cdot (\mathcal{J}^{\delta} y^{\delta}_{xx})^3}{(1+ (\mathcal{J}^{\delta} y^{\delta}_x)^2)^4}\bigg] dx.
	\end{eqnarray*}
	One can put in more effort and simplify other terms for a more convenient form, 
	but we avoid  long calculations and work with the following simplified version, as it is enough for our purpose:
	\begin{eqnarray}\label{en02}
	&&\hspace{7em} \f{1}{2} \partial_t \|\partial^5_x y^{\delta}\|^2_{L^2}+ \al^2(\al+ 3) \int   	 \f{(\mathcal{J}^{\delta} \partial^7_x y^{\delta})^2}{(1+ (\mathcal{J}^{\delta} y^{\delta}_x)^2)^2} dx \\ \nonumber
	&& \leq  |\al- 1| \cdot \bigg| \int ( \mathcal{J}^{\delta} \partial^7_x y^{\delta}) \cdot  \partial^3_x \bigg[\f{ (\mathcal{J}^{\delta} y^{\delta}_{xx})}{1+ (\mathcal{J}^{\delta} y^{\delta}_x)^2}\bigg] dx \bigg|+ 4 \al^2(\al+ 3) \bigg|\int ( \mathcal{J}^{\delta} \partial_x^7 y^\de) \cdot  \f{( \mathcal{J}^{\delta} \partial_x^6 y^\de) ( \mathcal{J}^{\delta} y^\de_{x}) (y^\de_{xx})}{(1+ (\mathcal{J}^{\delta} y^{\delta}_x)^2)^3} dx\bigg| \\ \nonumber
	&&+ 4 \al^2(\al+ 3) \bigg| \int ( \mathcal{J}^{\delta} \partial_x^7 y^\de) \cdot  \partial_x^2 \bigg[\f{( \mathcal{J}^{\delta} \partial_x^4 y^\de) ( \mathcal{J}^{\delta} y^\de_{x}) (y^\de_{xx})}{(1+ (\mathcal{J}^{\delta} y^{\delta}_x)^2)^3} \bigg] dx\bigg| +
	3 \al^2(\al+ 3) \bigg|\int  (\mathcal{J}^{\delta} \partial^7_x y^{\delta}) \cdot  \partial^3_x \bigg[\f{  (\mathcal{J}^{\delta} y^{\delta}_{xx})^3}{(1+ (\mathcal{J}^{\delta} y^{\delta}_x)^2)^3} \bigg] dx\bigg| \\ \nonumber
	&& + 10 \al^2(\al+ 3) \bigg| \int  (\mathcal{J}^{\delta} \partial^7_x y^{\delta}) \cdot  \partial^3_x \bigg[\f{   (\mathcal{J}^{\delta} y^{\delta}_x) (\mathcal{J}^{\delta} y^{\delta}_{xx}) (\mathcal{J}^{\delta} y^{\delta}_{xxx})}{(1+ (\mathcal{J}^{\delta} y^{\delta}_x)^2)^3} \bigg] dx\bigg|\\ \nonumber
	&&+ 18 \al^2(\al+ 3) \bigg|\int  (\mathcal{J}^{\delta} \partial^7_x y^{\delta}) \cdot  \partial^3_x \bigg[\f{ ( \mathcal{J}^{\delta} y^{\delta}_x)^2 (\mathcal{J}^{\delta} y^{\delta}_{xx})^3 }{(1+ (\mathcal{J}^{\delta} y^{\delta}_x)^2)^4}\bigg] dx\bigg|+  \bigg| \int (\mathcal{J}^{\delta} \partial^7_x y^{\delta}) \cdot  \partial^3_x  \bigg[\f{1+ (\mathcal{J}^{\delta}  y^{\delta}_x)^2}{\sqrt{1+ (\mathcal{J}^{\delta}  y^{\delta}_x)^2} }\bigg] dx \bigg|\\ \nonumber
	&&+ \bigg(1+ \f{1}{2} \al^2\bigg) \bigg| \int (\mathcal{J}^{\delta} \partial_x^7 y^{\delta}) \cdot \partial_x^2 \bigg[\f{(\mathcal{J} y^{\delta}_{xx})^2}{(1+ (\mathcal{J}^{\delta} y^{\delta}_x)^2)^{\f{5}{2}}} \bigg] dx \bigg|+  \bigg|\bigg(2 \al+ 5 \al^2- \f{1}{3} \al^3\bigg)  \int (\mathcal{J}^{\delta} \partial_x^7 y^{\delta}) \cdot \partial_x^2 \bigg[\f{(\mathcal{J} y^{\delta}_{xx})^3}{(1+ (\mathcal{J}^{\delta} y^{\delta}_x)^2)^4} \bigg] dx\bigg|\\ \nonumber
	&&  := I_1+ \cdots+ I_7+ I_8+ I_9.
	\end{eqnarray}
	We will omit most details of the estimates of these terms, as the proof is similar in many respects to the previous lemma.  For those details that we do show, in order to control $I_1, \cdots I_9$ we will focus on the worst term in each of them. In fact, the worse terms are the ones that the derivative behind the fractions hit the highest degree in the numerator. Note that as long as we restrict the time interval to the interval in \eqref{tt00}, Lemma \ref{lem00} already provides us with $H^4$ bounds, and hence, for $a= 0, 1, 2, 3$, there is a constant $C$ so that, 
	\begin{eqnarray}\label{es11}
	\|\partial_x^a y^{\de}\|_{L^{\infty}} \leq C.
	\end{eqnarray} 
	Therefore, any term of this kind which comes up in the estimates is easily bounded by a constant $C$. Moreover, an application of the Gagliardo-Nirenberg inequality and Lemma \ref{lem00} leads to
	\begin{eqnarray}\label{es110}
	\|\partial_x^4 y^{\de}\|_{L^{\infty}} \leq \eps_0 \|\partial_x^7 y^{\de}\|_{L^2}+ C \| y^{\de}\|_{L^2}
	\end{eqnarray}	
	 where in our future calculations, the constant $\eps_0$ will be chosen in a way that the seventh derivatives  on the right hand side could be absorbed in the left hand side 
	 (as is frequently done in energy estimates for parabolic equations).
	 This incurs the expense of a potentially large constant $C> 0$ on the term $\| y^{\de}\|_{L^2}$.

	We will start with the term $I_{1},$ and as mentioned above we only present the bound for the worst term in the expansion of $I_{1}:$ 
	\begin{eqnarray*}
		I_1 &=&  |\al- 1| \cdot \bigg| \int \f{( \mathcal{J}^{\delta} \partial^7_x y^{\delta})}{1+ (\mathcal{J}^{\delta} y^{\delta}_x)^2 }  \cdot \bigg(1+ (\mathcal{J}^{\delta} y^{\delta}_x)^2 \bigg) \cdot  \partial^3_x \bigg[\f{ (\mathcal{J}^{\delta} y^{\delta}_{xx})}{1+ (\mathcal{J}^{\delta} y^{\delta}_x)^2}\bigg] dx \bigg| \\
		&\leq&   C \bigg\| \f{( \mathcal{J}^{\delta} \partial^7_x y^{\delta})}{1+ (\mathcal{J}^{\delta} y^{\delta}_x)^2}  \bigg\|_{L^2} \bigg\|1+ (\mathcal{J}^{\delta} y^{\delta}_x)^2 \bigg\|_{L^{\infty}} \bigg\| \partial^3_x \bigg[\f{ (\mathcal{J}^{\delta} y^{\delta}_{xx})}{1+ (\mathcal{J}^{\delta} y^{\delta}_x)^2}\bigg] \bigg\|_{L^2} \leq C \bigg\| \f{( \mathcal{J}^{\delta} \partial^7_x y^{\delta})}{1+ (\mathcal{J}^{\delta} y^{\delta}_x)^2}   \bigg\|_{L^2}  \cdot  \bigg\|\partial^3_x \bigg[\f{ (\mathcal{J}^{\delta} y^{\delta}_{xx})}{1+ (\mathcal{J}^{\delta} y^{\delta}_x)^2} \bigg] \bigg\|_{L^2}\\
		&\leq& C \bigg\| \f{( \mathcal{J}^{\delta} \partial^7_x y^{\delta})}{1+ (\mathcal{J}^{\delta} y^{\delta}_x)^2}   \bigg\|_{L^2}  \cdot  \bigg\| \f{ (\mathcal{J}^{\delta} \partial^5_x y^{\delta})}{1+ (\mathcal{J}^{\delta} y^{\delta}_x)^2} \bigg\|_{L^2}+ G \leq C \bigg\| \f{( \mathcal{J}^{\delta} \partial^7_x y^{\delta})}{1+ (\mathcal{J}^{\delta} y^{\delta}_x)^2}   \bigg\|_{L^2}  \cdot  \bigg(\bigg\| \mathcal{J}^{\delta} \partial^5_x y^{\delta}\bigg\|_{L^2}+ G\bigg)\\
		&\leq& C \bigg\| \f{( \mathcal{J}^{\delta} \partial^7_x y^{\delta})}{1+ (\mathcal{J}^{\delta} y^{\delta}_x)^2}   \bigg\|_{L^2}  \cdot  \bigg(\bigg\| \mathcal{J}^{\delta} \partial^4_x y^{\delta}\bigg\|^{\f{2}{3}}_{L^2} \bigg\| \mathcal{J}^{\delta} \partial^7_x y^{\delta}\bigg\|^{\f{1}{3}}_{L^2}\bigg)+ G \leq  \bigg\| \f{( \mathcal{J}^{\delta} \partial^7_x y^{\delta})}{1+ (\mathcal{J}^{\delta} y^{\delta}_x)^2}\bigg\|^{\f{4}{3}}_{L^2} \bigg\|1+ (\mathcal{J}^{\delta} y^{\delta}_x)^2 \bigg\|^{\f{1}{3}}_{L^{\infty}}\\
		&\leq& \f{1}{10} \bigg\| \f{( \mathcal{J}^{\delta} \partial^7_x y^{\delta})}{1+ (\mathcal{J}^{\delta} y^{\delta}_x)^2}\bigg\|^2_{L^2}+ C. 
	\end{eqnarray*}
	
	We omit the details for $I_{2},$ but reach the same conclusion as for $I_{1}.$

The estimate for $I_{3}$ has an interesting feature which we mention.  To begin, we have
	\begin{eqnarray*}
		I_3 &\leq&  |\al^2(\al+ 3)| \bigg| \int ( \mathcal{J}^{\delta} \partial_x^7 y^\de) \cdot  \partial_x^2 \bigg[\f{( \mathcal{J}^{\delta} \partial_x^4 y^\de) ( \mathcal{J}^{\delta} y^\de_{x}) (y^\de_{xx})}{(1+ (\mathcal{J}^{\delta} y^{\delta}_x)^2)^3} \bigg] dx\bigg|\\
		& \leq& C \bigg| \int ( \mathcal{J}^{\delta} \partial_x^7 y^\de) \cdot  \bigg[\f{( \mathcal{J}^{\delta} \partial_x^6 y^\de) ( \mathcal{J}^{\delta} y^\de_{x}) (y^\de_{xx})}{(1+ (\mathcal{J}^{\delta} y^{\delta}_x)^2)^3} \bigg] dx\bigg|+ G.
	\end{eqnarray*}
	The integral on the right hand side of this may be controlled as desired. 
	The interesting feature mentioned above has to do with an estimate of a lower-order term from the collection $G,$ namely
	\begin{eqnarray*}
	\bigg| \int ( \mathcal{J}^{\delta} \partial_x^7 y^\de) \cdot  \bigg[\f{( \mathcal{J}^{\delta} \partial_x^4 y^\de)^2 ( \mathcal{J}^{\delta} y^\de_{x})}{(1+ (\mathcal{J}^{\delta} y^{\delta}_x)^2)^3} \bigg] dx\bigg| &\leq& \|\mathcal{J}^{\delta} \partial_x^7 y^\de\|_{L^2} \|\mathcal{J}^{\delta} \partial_x^4 y^\de\|_{L^{2}}  \|\mathcal{J}^{\delta} \partial_x^4 y^\de\|_{L^{\infty}} \\
	&\leq& \f{1}{10} \left\|\frac{\mathcal{J}^{\delta} \partial_x^7 y^\de}{1+(\mathcal{J}^{\delta}y^{\delta}_{x})^{2}}\right\|^2_{L^2}+ C\|\mathcal{J}^{\delta} y^\de\|^2_{L^{2}}   \\
	&\leq& \f{1}{10} \left\|\frac{\mathcal{J}^{\delta} \partial_x^7 y^\de}{1+(\mathcal{J}^{\delta}y^{\delta}_{x})^{2}}\right\|^2_{L^2}+ C
	\end{eqnarray*}
in which we have used \eqref{es110} and Lemma \ref{lem00}.

	Omitting further details, we have the conclusion
	\begin{eqnarray*}
		&& \f{1}{2} \partial_t \|\partial^5_x y^{\delta}\|^2_{L^2}+ C_0 \int   	 \f{(\mathcal{J}^{\delta} \partial^7_x y^{\delta})^2}{(1+ (\mathcal{J}^{\delta} y^{\delta}_x)^2)^2} dx \leq C.
	\end{eqnarray*}
	where $C_0$ is a positive constant and satisfies $C_0 \leq 4- \f{9}{10}$. This  clearly finishes the proof.
\end{proof}

Now we are ready to present the existence of the solution to the initial value problem for 
 \eqref{yt} in the Sobolev space $H^5$. 
\label{sec:4}
\begin{lem}\label{lem0}
	For all $0 < t< \f{1}{\ga}\ln\bigg(1+ \f{\ga}{\|y_0\|^{m- 2}_{H^4}}\bigg)$ there exists a function $y \in H^5$ that solves the equation \eqref{yt}, with initial data $y(\cdot,0)=y_{0}\in H^{5}.$ Moreover, there is a constant $C= C(y_0, \al)$ so that 
	\begin{eqnarray}
	\sup_{0< t< T} \|y\|_{H^5} \leq C.
	\end{eqnarray}
\end{lem}
\begin{proof} 
	In Lemma \ref{lem00} we have shown that $\{y^{\de}\}_{\de> 0}$ is a uniformly bounded  and  continuous  family of  functions defined  on $\mathbb{R} \times [0, T]$ in the Sobolev space $H^5$. Hence, an application of the Banach-Aloaglu provides a subset $\{y^{\de_j}\}_j$ and a function $y \in H^5$ so that $y^{\de_j} \to y$ in $H^5$. We claim that this $y \in H^5$ solves the equation \eqref{yt}. Indeed, for all $\de> 0$, the integral representation of the solution to the equation \eqref{yt0} is in hand,
	\begin{eqnarray}\label{y0}
	y^{\de}(x, t)= y_0^{\de}(x)+ \int_0^t F(y^{\de})(x,s) ds.
	\end{eqnarray}
	where $F(\cdot)$ is defined in  \eqref{FF}.  The integrand in the right hand side consists of continuous terms with functions $\partial_x^s y^{\de}$, $0 \leq s \leq 4$ within. Therefore in \eqref{y0}, there is no difficulty in passing to limit on the subsequence $\de_j$. Since $y^{\de_j} \to y$ as $j \to \infty$, that means
	\begin{eqnarray}\label{y1}
	y(x, t)= y_0(x)+ \int_0^t F(y)(s) ds.
	\end{eqnarray}
	This immediately implies that \eqref{y1} satisfies the equation \eqref{yt}. 
\end{proof}


\section{Asymptotics}
\label{sec:5}
In this section we show that, in a special scaling limit, solutions of the system \eqref{yt} and solutions of the KS
equation \eqref{UU} shadow one another over a time period dependent on initial values of both equations. 
To begin, we fix an $0< \ep \ll 1$ and assume
$$
\al= 1+\eps.
$$
Then we use the change of variables 
\begin{equation}\label{sc}
(\xi, \tau)= (\eps^\f{1}{2} x, \eps^2 t),\ \ \ y(x, t)= \eps  \Phi(\sqrt{\eps} x, \eps^2 t)- t.
\end{equation}
A straightforward calculation  transfers the equation \eqref{yt} into new variables as follows, 
\begin{eqnarray*}
	&&y_t= \eps^3 \Phi_{\tau}- 1,\ \ y_x= \eps^{\f{3}{2}} \Phi_{\xi},\ \ \f{(\al- 1) y_{x x}}{1+ y_x^2}= \f{\eps^3 \Phi_{\xi \xi}}{1+ \eps^3(\Phi_{\xi})^2}  \\
	&& \sqrt{1+  ( y_x)^2} = \sqrt{1+ \eps^3(\Phi_{\xi})^2}, \\
	&& \f{1}{\sqrt{1+ y_x^2}} \cdot \f{d^2 \kappa}{d x^2}= \f{\eps^3 \Phi_{\xi \xi \xi \xi}}{\bigg(1+  (\eps^{\f{3}{2}} \Phi_{\xi})^2\bigg)^2}- \f{3 \eps^6 (\Phi_{\xi \xi})^3+ 9 \eps^6 \Phi_{\xi} \Phi_{\xi \xi} \Phi_{\xi \xi \xi}}{\bigg(1+  (\eps^{\f{3}{2}} \Phi_{\xi})^2\bigg)^3 }+ \f{15 \eps^{9} (\Phi_{\xi})^2 (\Phi_{\xi \xi})^3 }{\bigg(1+  (\eps^{\f{3}{2}} \Phi_{\xi})^2\bigg)^4 }
\end{eqnarray*}
also,
\begin{eqnarray*}
	y_x \cdot \kappa \cdot \f{d \kappa}{d x}&=& \f{y_x y_{xx} y_{xxx}}{(1+ y_x^2)^3}- \f{3  (y_x)^2 (y_{xx})^3}{(1+ y_x^2)^4}=  \eps^6\f{\Phi_{\xi} \Phi_{ \xi \xi} \Phi_{ \xi \xi \xi}}{(1+ \eps^3 (\Phi_{\xi})^2)^3}- \f{3 \eps^9 (\Phi_{\xi})^2 (\Phi_{ \xi \xi})^3}{(1+ \eps^3 (\Phi_{\xi})^2)^4}
\end{eqnarray*}
Then
\begin{eqnarray*}
	\eps^3 \Phi_{\tau}&+&   \f{\eps^3  \Phi_{\xi \xi}}{1+ \eps^3 (\Phi_{\xi})^2}+ \al^2(\al+ 3) \f{\eps^3 \Phi_{\xi \xi \xi \xi}}{(1+ \eps^3 (\Phi_{\xi})^2)^2}+  \sqrt{1+ \eps^3 (\Phi_{\xi})^2}- 1 =    10 \al^2(\al+ 3) \f{\eps^6 \Phi_{\xi} \Phi_{\xi \xi} \Phi_{\xi \xi \xi}}{(1+ \eps^3 (\Phi_{\xi})^2)^3}\\
	&+&  3 \al^2(\al+ 3) \f{\eps^6 (\Phi_{\xi \xi})^3}{(1+ \eps^3 (\Phi_{\xi})^2)^3}-   18 \al^2(\al+ 3) \f{\eps^9 (\Phi_{\xi \xi})^3 (\Phi_{\xi})^2}{(1+ \eps^3 (\Phi_{\xi})^2)^4}- \bigg(1+ \f{1}{2} \al^2\bigg) \f{\eps^4 (\Phi_{\xi \xi})^2}{(1+ \eps^3 (\Phi_{\xi})^2)^{\f{5}{2}}}\\
	&-& \bigg(2 \al+ 5 \al^2- \f{1}{3} \al^3\bigg)  \f{\eps^6 (\Phi_{\xi \xi})^3}{(1+ \eps^3 (\Phi_{\xi})^2)^4}.
\end{eqnarray*}
This leads to 
\begin{eqnarray}\label{Phi0} 
\begin{cases}
\Phi_{\tau}+ \al^2(\al+ 3) \f{ \Phi_{\xi \xi \xi \xi}}{(1+ \eps^3 (\Phi_{\xi})^2)^2}= -   \f{ \Phi_{\xi \xi}}{1+ \eps^3 (\Phi_{\xi})^2} - \frac{(\Phi_{ \xi })^2}{1+ \sqrt{1+ \eps^3 (\Phi_{\xi})^2}}+    10 \al^2(\al+ 3) \f{\eps^3 \Phi_{\xi} \Phi_{\xi \xi} \Phi_{\xi \xi \xi}}{(1+ \eps^3 (\Phi_{\xi})^2)^3}+
\\ 
\hspace{2em}+  3 \al^2(\al+ 3) \f{\eps^3 (\Phi_{\xi \xi})^3}{(1+ \eps^3 (\Phi_{\xi})^2)^3}-   18 \al^2(\al+ 3) \f{\eps^6 (\Phi_{\xi \xi})^3 (\Phi_{\xi})^2}{(1+ \eps^3 (\Phi_{\xi})^2)^4}- \bigg(1+ \f{1}{2} \al^2\bigg) \f{\eps (\Phi_{\xi \xi})^2}{(1+ \eps^3 (\Phi_{\xi})^2)^{\f{5}{2}}}-\\
\hspace{2em}- \bigg(2 \al+ 5 \al^2- \f{1}{3} \al^3\bigg)  \f{\eps^3 (\Phi_{\xi \xi})^3}{(1+ \eps^3 (\Phi_{\xi})^2)^4}.
\end{cases}
\end{eqnarray}
In the above, if we put $\ep = 1$ and $\alpha=1$ we arrive at
$$
\Phi_\tau + 4 \Phi_{\xi\xi\xi\xi} = - \Phi_{\xi \xi} -{1\over 2} \Phi_\xi^2
$$
which is just \eqref{UU} with a new variable name. It is worth noting that 
putting $f(x, t)= \eps  U(\sqrt{\eps} x, \eps^2 t)$
 transfers \eqref{ff} into equation \eqref{UU} as well.
 The point of this section  is to make a rigorous the comparison of solutions of \eqref{Phi0}
 to those of \eqref{UU} when $\ep$ is small.
  
\subsection{Some \emph{apriori} estimate for the function $\Phi(\xi, \tau)$}
We now turn our attention to some bounds for the solution $\Phi$ of \eqref{Phi0} in Sobolev spaces. 
Specifically we have the following lemma:
\begin{lem}\label{lem3}
	Fix $\tau_*$ and $\Ga_*$. Then there exists constants $E_*$ and $\eps_*$ so that if $\|\Phi(0)\|_{H^4} \leq E_*$ and $0< \eps< \eps_*$ and if   $|\al- 1|= \eps$, then  
	\begin{eqnarray}
	\sup_{0< \tau< \tau_*} \|\Phi(\tau)\|_{H^4} \leq \Ga_*.
	\end{eqnarray}
\end{lem}

Before the proof, note that this lemma tells us that after unraveling the scaling from \eqref{sc} to go from $\Phi$
back to $y$, we find that the solution $y(x,t)$ exists on the time interval $[0,\tau_*/\ep^2]$, far longer than the times of existence we found in the previous section.

\begin{proof}
	The proof goes by adding up two energy estimates together, one on $\|\Phi\|_{L^2}$ and the other one on $\|\partial_x^4 \Phi\|_{L^2}$. We first multiply the equation \eqref{Phi0} into $\Phi$ and take the integral to get the following energy estimate:
	\begin{eqnarray*}
	&&\frac{1}{2} \partial_t \|\Phi\|_{L^2}^2+ \al^2(\al+ 3) \|\f{\Phi_{ \xi \xi } }{1+ \eps^3 (\phi_{\xi})^2}\|_{L^2}^2=   -  \int  \f{\Phi_{ \xi \xi } \Phi}{1+ \eps^3 (\phi_{\xi})^2} d\xi-  \int  \frac{\Phi (\Phi_{ \xi })^2}{1+ \sqrt{1+ \eps^3 (\Phi_{\xi})^2}} d\xi\\ \nonumber
	&-&   \bigg(1+ \f{1}{2} \al^2\bigg)  \eps \int \f{\Phi (\Phi_{\xi \xi})^2}{(1+ \eps^3 (\Phi_{\xi})^2)^{\f{5}{2}}} d\xi+ 10 \al^2(\al+ 3) \eps^3 \int \f{\Phi \Phi_{\xi} \Phi_{\xi \xi} \Phi_{\xi \xi \xi} }{(1+ \eps^3 (\phi_{\xi})^2)^3}d\xi+  4 \al^2(\al+ 3)  \eps^3 \int \f{ (\Phi_{\xi \xi})^2 (\Phi_{\xi})^2 }{(1+ \eps^3 (\phi_{\xi})^2)^3 }d\xi\\ \nonumber
	&-& 4 \eps^3 \al^2(\al+ 3) \int \Phi_{ \xi \xi } \p_{\xi} \bigg[\f{\Phi \Phi_{ \xi } \Phi_{ \xi \xi }}{(1+ \eps^3 (\Phi_{ \xi })^2)^3}\bigg] d\xi +3 \al^2(\al+ 3) \eps^3 \int \f{ (\Phi_{\xi \xi})^3 \Phi }{(1+ \eps^3 (\phi_{\xi})^2)^3}d\xi\\ \nonumber
	&-& \bigg(2 \al+ 5 \al^2- \f{1}{3} \al^3\bigg) \eps^3  \int \f{ \Phi (\Phi_{\xi \xi})^3}{(1+ \eps^3 (\Phi_{\xi})^2)^4}-  18 \al^2(\al+ 3) \eps^6 \int \f{ (\Phi_{\xi \xi})^3 (\Phi_{\xi})^2 \Phi}{(1+ \eps^3 (\phi_{\xi})^2)^4}d\xi\\
	& :=& I_1+ I_2+ I_3+ I_4+ I_5+ I_6+ I_7+ I_8+ I_9.
	\end{eqnarray*}
Above we made a simplification on one of the integrals using integration by parts, 
\begin{eqnarray*}
\int  \f{\p^4_{\xi} \Phi \cdot  \Phi  }{(1+ \eps^3 (\phi_{\xi})^2 )^2} d\xi= \int  \f{(\p_{\xi \xi} \Phi )^2   }{(1+ \eps^3 (\phi_{\xi})^2 )^2} d\xi-  4   \eps^3 \int \f{ (\Phi_{\xi \xi})^2 (\Phi_{\xi})^2 }{(1+ \eps^3 (\phi_{\xi})^2)^3 }d\xi-  4 \eps^3 \int \Phi_{ \xi \xi } \p_{\xi} \bigg[\f{\Phi \Phi_{ \xi } \Phi_{ \xi \xi }}{(1+ \eps^3 (\Phi_{ \xi })^2)^3}\bigg].
\end{eqnarray*}
	Although this energy estimate is already in a nice form, and we can run the argument, we are still able to simplify the left hand side, which itself reduces many calculations. In fact we make use of the bounds in Lemma \ref{lem0}, and the scaling \eqref{sc}  
	\begin{eqnarray}\label{b0}
	\bigg\|\f{1}{1+ \eps^3 (\Phi_{ \xi })^2}\bigg\|_{L^{\infty}}= \bigg\|\f{1}{1+ (y_x)^2}\bigg\|_{L^{\infty}} \geq C_0
	\end{eqnarray}
	for some $C_0> 0$ fixed. Therefore, the above energy estimates turns into,
\begin{eqnarray}\label{es7}
\frac{1}{2} \partial_t \|\Phi\|_{L^2}^2+ C \al^2(\al+ 3) \|\Phi_{ \xi \xi }\|_{L^2}^2 \leq \bigg|I_1+ I_2+ I_3+ I_4+ I_5+ I_6+ I_7+ I_8\bigg|.
\end{eqnarray}
	We now try to find a proper bound for the right hand side of this equality. \\
	\textbf{Estimate for $I_1$:} 
	\begin{eqnarray*}
		|I_1|=  \bigg| \int  \f{\Phi_{ \xi \xi } \Phi}{1+ \eps^3 (\phi_{\xi})^2} d\xi\bigg| \leq C \|\Phi\|_{L^2} \|\Phi_{\xi \xi}\|_{L^2} \leq  \|\Phi\|_{L^2} \bigg(\|\Phi\|^{\f{1}{2}}_{L^2} \|\partial^4 \Phi\|^{\f{1}{2}}_{L^2} \bigg) \leq C \bigg(\|\Phi\|^2_{L^2}+ \|\partial^4 \Phi\|^2_{L^2} \bigg)
	\end{eqnarray*}
	\textbf{Estimate for $I_2$:} 
	\begin{eqnarray*}
		|I_2|&=&  \bigg|  \int  \frac{\Phi (\Phi_{ \xi })^2}{1+ \sqrt{1+ \eps^3 (\Phi_{\xi})^2}} d\xi\bigg| \leq \bigg\| \frac{1}{1+ \sqrt{1+ \eps^3 (\Phi_{\xi})^2}}\bigg\|_{L^{\infty}} \|\Phi\|_{L^{\infty}} \|\Phi_{\xi}\|^2_{L^2} \leq C  \||\nabla|^{\f{1}{2}}\Phi\|_{L^2} \|\Phi_{\xi}\|^2_{L^2}\\
		& \leq& C \bigg(\|\Phi\|^{\f{7}{8}}_{L^2} \|\partial^4 \Phi\|^{\f{1}{8}}_{L^2} \bigg) \bigg(\|\Phi\|^{\f{3}{4}}_{L^2} \|\partial^4_{\xi} \Phi\|^{\f{1}{4}}_{L^2} \bigg)^2 \leq  C \bigg(\|\Phi\|^2_{L^2}+ \|\partial^4_{\xi} \Phi\|^2_{L^2} \bigg)+ C \bigg(\|\Phi\|^4_{L^2}+ \|\partial^4_{\xi} \Phi\|^4_{L^2} \bigg).
	\end{eqnarray*}

	\textbf{Estimate for $I_3$:} 
	\begin{eqnarray*}
		|I_3|&=&   \eps  \bigg(1+ \f{1}{2} \al^2\bigg)  \bigg| \int \f{\Phi (\Phi_{\xi \xi})^2}{(1+ \eps^3 (\Phi_{\xi})^2)^{\f{5}{2}}} d\xi\bigg| \leq  \eps \|\Phi\|_{L^2} \|\Phi_{\xi \xi}\|^2_{L^2}  \bigg\| \frac{1}{ (1+ \eps^3 (\Phi_{\xi})^2)^{\f{5}{2}}}\bigg\|_{L^{\infty}}\\
		& \leq& \eps \|\Phi\|_{L^2} \bigg(\|\Phi\|^{\f{1}{2}}_{L^2} \|\partial^4 \Phi\|^{\f{1}{2}}_{L^2} \bigg)^2 
		\leq C \eps \bigg(\|\Phi\|^3_{L^2}+ \|\partial^4 \Phi\|^3_{L^2}\bigg) \\ 
		&\leq& \bigg( \|\Phi\|^2_{L^2}+ \|\partial^4 \Phi\|^2_{L^2}  \bigg)+ C \eps^2 \bigg( \|\Phi\|^4_{L^2}+ \|\partial^4 \Phi\|^4_{L^2}  \bigg).
	\end{eqnarray*}

	\textbf{Estimate for $I_4$:} 
	\begin{eqnarray*}
		|I_4|&=&    6 \al^2(\al+ 3) \eps^3 \bigg|  \int \f{\Phi \Phi_{\xi} \Phi_{\xi \xi} \Phi_{\xi \xi \xi} }{(1+ \eps^3 (\phi_{\xi})^2)^3}d\xi\bigg| \leq \eps^3 \bigg\| \frac{1}{ (1+ \eps^3 (\Phi_{\xi})^2)^3}\bigg\|_{L^{\infty}} \|\Phi\|_{L^{\infty}} \|\Phi_{ \xi }\|_{L^{\infty}}  \|\Phi_{\xi \xi}\|_{L^2} \|\Phi_{\xi \xi \xi}\|_{L^2}  \\
		&\leq& C \eps^3  \||\nabla|^{\f{1}{2}} \Phi\|_{L^2} \||\nabla|^{\f{1}{2}} \Phi_{\xi}\|_{L^2} \|\Phi_{\xi \xi}\|_{L^2} \|\Phi_{\xi \xi \xi}\|_{L^2} \\
		&\leq& C \eps^3  \bigg(\|\Phi\|^{\f{7}{8}}_{L^2} \|\partial_{\xi}^4 \Phi\|^{\f{1}{8}}_{L^2} \bigg)  \bigg(\|\Phi\|^{\f{5}{8}}_{L^2} \|\partial_{\xi}^4 \Phi\|^{\f{3}{8}}_{L^2} \bigg)   \bigg(\|\Phi\|^{\f{1}{2}}_{L^2} \|\partial_{\xi}^4 \Phi\|^{\f{1}{2}}_{L^2} \bigg) \bigg(\|\Phi\|^{\f{1}{4}}_{L^2} \|\partial_{\xi}^4 \Phi\|^{\f{3}{4}}_{L^2} \bigg)\\
		& \leq& C \eps^3   \bigg(\|\Phi\|^4_{L^2}+ \|\partial_{\xi}^4 \Phi\|^4_{L^2} \bigg).
	\end{eqnarray*}
	\textbf{Estimate for $I_5$:} 
	\begin{eqnarray*}
	|I_5| &\leq& 4 \eps^3 \al^2(\al+ 3) \bigg|\int \Phi_{ \xi \xi } \p_{\xi} \bigg[\f{\Phi \Phi_{ \xi } \Phi_{ \xi \xi }}{(1+ \eps^3 (\Phi_{ \xi })^2)^3}\bigg] d\xi\bigg|\\
	& \leq& C \eps^3 \|\Phi_{ \xi \xi }\|_{L^2}  \|\Phi_{ \xi \xi \xi}\|_{L^2} \|\Phi\|_{L^{\infty}} \|\Phi_{\xi}\|_{L^{\infty}} \| \f{1}{1+ \eps^3 (\Phi_{ \xi })^2}\|_{L^{\infty}}\\
	&\leq& C \eps^3 \|\Phi_{ \xi \xi }\|_{L^2}  \|\Phi_{ \xi \xi \xi}\|_{L^2} \||\nabla |^{\f{1}{2}} \Phi\|_{L^2} \||\nabla |^{\f{3}{2}} \Phi\|_{L^2}\\
	& \leq& C \eps^3 \bigg(\|\Phi\|^{\f{1}{2}}_{L^2} \|\partial_{\xi}^4 \Phi\|^{\f{1}{2}}_{L^2} \bigg) \bigg(\|\Phi\|^{\f{1}{4}}_{L^2} \|\partial_{\xi}^4 \Phi\|^{\f{3}{4}}_{L^2} \bigg)  \bigg(\|\Phi\|^{\f{7}{8}}_{L^2} \|\partial_{\xi}^4 \Phi\|^{\f{1}{8}}_{L^2} \bigg) \bigg(\|\Phi\|^{\f{5}{8}}_{L^2} \|\partial^4 \Phi\|^{\f{3}{8}}_{L^2} \bigg)\\
	&\leq& C \eps^3 \bigg(\|\Phi\|^4_{L^2}+ \|\partial_{\xi}^4 \Phi\|^4_{L^2} \bigg).
	\end{eqnarray*}
	\textbf{Estimate for $I_6$:} 
	\begin{eqnarray*}
		|I_6|&=&    4 \al^2(\al+ 3)  \eps^3 \bigg|  \int \f{ (\Phi_{\xi \xi})^2 (\Phi_{\xi})^2 }{(1+ \eps^3 (\phi_{\xi})^2)^3 }d\xi\bigg| \leq \eps^3 \bigg\| \frac{1}{(1+ \eps^3 (\Phi_{\xi})^2)^3}\bigg\|_{L^{\infty}} \|\Phi_{\xi \xi}\|^2_{L^2} \|\Phi_{\xi}\|^2_{L^{\infty}}   \\
		&\leq& C \eps^3   \|\Phi_{\xi \xi}\|^2_{L^2} \||\nabla|^{\f{1}{2}} \Phi_{\xi}\|^2_{L^2} \leq C \eps^3   \bigg(\|\Phi\|^{\f{1}{2}}_{L^2} \|\partial^4 \Phi\|^{\f{1}{2}}_{L^2} \bigg)^2 \bigg(\|\Phi\|^{\f{5}{8}}_{L^2} \|\partial^4 \Phi\|^{\f{3}{8}}_{L^2} \bigg)^2 \leq \eps^3   \bigg(\|\Phi\|^4_{L^2}+ \|\partial^4 \Phi\|^4_{L^2} \bigg).
	\end{eqnarray*}
	\textbf{Estimate for $I_7$:} 
	\begin{eqnarray*}
		|I_6|&=&    3 \al^2(\al+ 3)  \eps^3 \bigg|  \int \f{ (\Phi_{\xi \xi})^3 \Phi }{(1+ \eps^3 (\phi_{\xi})^2)^3}d\xi\bigg| \leq C \eps^3 \bigg\| \frac{1}{(1+ \eps^3 (\Phi_{\xi})^2)^3}\bigg\|_{L^{\infty}} \|\Phi_{\xi \xi}\|^3_{L^3} \|\Phi\|_{L^{\infty}}   \\
		&\leq& C \eps^3   \||\nabla|^{\f{13}{6}}\Phi\|^3_{L^2} \||\nabla|^{\f{1}{2}} \Phi\|_{L^2} \leq C \eps^3   \bigg(\|\Phi\|^{\f{11}{24}}_{L^2} \|\partial^4 \Phi\|^{\f{13}{24}}_{L^2} \bigg)^3 \bigg(\|\Phi\|^{\f{7}{8}}_{L^2} \|\partial^4 \Phi\|^{\f{1}{8}}_{L^2} \bigg) \leq C \eps^3   \bigg(\|\Phi\|^4_{L^2}+ \|\partial^4 \Phi\|^4_{L^2} \bigg).
	\end{eqnarray*}
	\textbf{Estimate for $I_8$:} Similar to $I_7$.\\
	\textbf{Estimate for $I_9$:} 
	\begin{eqnarray*}
		|I_9|&=&   18 \al^2(\al+ 3) \eps^6 \bigg|  \int \f{ (\Phi_{\xi \xi})^3 (\Phi_{\xi})^2 \Phi}{(1+ \eps^3 (\phi_{\xi})^2)^4}d\xi\bigg| \leq C \bigg\| \frac{ \eps^3 (\Phi_{\xi})^2}{(1+ \eps^3 (\Phi_{\xi})^2)^4}\bigg\|_{L^{\infty}} \bigg( \eps^3 \|\Phi_{\xi \xi}\|^3_{L^3} \|\Phi\|_{L^{\infty}} \bigg)  \\
		&\leq& C \eps^3   \||\nabla|^{\f{13}{6}}\Phi\|^3_{L^2} \||\nabla|^{\f{1}{2}} \Phi\|_{L^2} \leq C \eps^3   \bigg(\|\Phi\|^{\f{11}{24}}_{L^2} \|\partial_{\xi}^4 \Phi\|^{\f{13}{24}}_{L^2} \bigg)^3 \bigg(\|\Phi\|^{\f{7}{8}}_{L^2} \|\partial_{\xi}^4 \Phi\|^{\f{1}{8}}_{L^2} \bigg) \leq C \eps^3   \bigg(\|\Phi\|^4_{L^2}+ \|\partial_{\xi}^4 \Phi\|^4_{L^2} \bigg).
	\end{eqnarray*} 
Overall, the energy estimates \eqref{es7} is transfered into
\begin{eqnarray}\label{gr0}
\f{1}{2} \partial_t \|\Phi\|^2_{L^2} \leq C \bigg( \|\Phi\|^2_{L^2}+ \|\partial_{\xi}^4 \Phi\|^2_{L^2}  \bigg)+ C \bigg( \|\Phi\|^4_{L^2}+ \|\partial_{\xi}^4 \Phi\|^4_{L^2}  \bigg).
\end{eqnarray}
	
	As it was mentioned before, our argument is based on a combined energy estimate. For the other term in the energy estimate, we take $4$ times derivative of the equation \eqref{Phi0}, and then find the inner product of the resulting equation with $\partial_{\xi}^4 \Phi$, 
	\begin{eqnarray}\label{en0}
	&&\frac{1}{2} \partial_t \|\partial_{\xi}^4 \Phi\|_{L^2}^2+ \al^2(\al+ 3)  \int  \partial_{\xi}^4 [\f{\partial_\xi^4 \Phi}{(1+ \eps^3 (\Phi_{ \xi })^2 )^2}] \partial_{\xi}^4 \Phi d\xi= -   \int  \partial_{\xi}^4 [\f{ \Phi_{\xi \xi}}{1+ \eps^3 (\Phi_{ \xi })^2 }] \partial_{\xi}^4 \Phi d\xi\\ \nonumber
	&-&      \int  \partial_{\xi}^4 [\frac{(\Phi_{ \xi })^2}{1+ \sqrt{1+ \eps^3 (\Phi_{\xi})^2}} ] \partial_{\xi}^4 \Phi d\xi+   10 \al^2(\al+ 3) \eps^3 \int \partial^4_{\xi} \bigg[\f{ \Phi_{\xi} \Phi_{\xi \xi} \Phi_{\xi \xi \xi} }{(1+ \eps^3 (\phi_{\xi})^2)^3} \bigg] \partial^4_{\xi} \Phi d\xi\\ \nonumber
	&+&  3 \al^2(\al+ 3) \eps^3 \int \partial^4_{\xi} \bigg[\f{ (\Phi_{\xi \xi})^3  }{(1+ \eps^3 (\phi_{\xi})^2)^3} \bigg] \partial^4_{\xi} \Phi d\xi- 18 \al^2(\al+ 3) \eps^6 \int \partial^4_{\xi} \bigg[\f{  (\Phi_{\xi \xi})^3 (\Phi_{\xi})^2 }{(1+ \eps^3 (\phi_{\xi})^2)^3} \bigg] \partial^4_{\xi} \Phi d\xi\\ \nonumber
	&-& \bigg(1+ \f{1}{2} \al^2\bigg)  \eps \int \p_{\xi}^4 \bigg[ \f{ (\Phi_{\xi \xi})^2}{(1+ \eps^3 (\Phi_{\xi})^2)^{\f{5}{2}}} \bigg] \p_{\xi}^4 \Phi d\xi - \bigg(2 \al+ 5 \al^2- \f{1}{3} \al^3\bigg) \eps^3  \int \p_{\xi}^4 \bigg[ \f{ (\Phi_{\xi \xi})^3}{(1+ \eps^3 (\Phi_{\xi})^2)^4} \bigg] \p_{\xi}^4 \Phi d\xi.
	\end{eqnarray}
	Although we can simplify most of the terms in this relation, we let most of them in the current form, as they are easily bounded in the current form. However, 
	\begin{eqnarray*}
		&&\al^2(\al+ 3)  \int  \partial_{\xi}^4 [\f{\partial_\xi^4 \Phi}{(1+ \eps^3 (\Phi_{ \xi })^2 )^2}] \partial_{\xi}^4 \Phi d\xi= \al^2(\al+ 3)  \int  \f{(\partial_\xi^6 \Phi)^2}{(1+ \eps^3 (\Phi_{ \xi })^2 )^2} d\xi\\
		&&- 4 \eps^3 \al^2 (\al+ 3) \int  \f{(\partial_{\xi}^6 \Phi) (\partial_\xi^5 \Phi) \Phi_{ \xi } \Phi_{ \xi \xi } }{(1+ \eps^3 (\Phi_{ \xi })^2 )^3} d\xi
		-  4 \eps^3 \al^2 (\al+ 3) \int (\partial_{\xi}^6 \Phi)  (\partial_\xi^4 \Phi) \p_{\xi} \bigg[\f{ \Phi_{ \xi } \Phi_{ \xi \xi } }{(1+ \eps^3 (\Phi_{ \xi })^2 )^3} \bigg] d\xi.
	\end{eqnarray*}
	Then, the energy estimates \eqref{en0} turns into
	\begin{eqnarray}\label{en11}
	&&\frac{1}{2} \partial_t \|\partial_{\xi}^4 \Phi\|_{L^2}^2+ \al^2(\al+ 3)  \int  \f{(\partial_\xi^6 \Phi)^2}{(1+ \eps^3 (\Phi_{ \xi })^2 )^2} d\xi= -   \int \partial_{\xi}^6 \Phi \cdot  \partial_{\xi}^2 \bigg[\f{ \Phi_{\xi \xi}}{1+ \eps^3 (\Phi_{ \xi })^2 }\bigg] d\xi\\ \nonumber
	&-&      \int \partial_{\xi}^6 \Phi \cdot \partial_{\xi}^2 \bigg[\frac{(\Phi_{ \xi })^2}{1+ \sqrt{1+ \eps^3 (\Phi_{\xi})^2}} \bigg]  d\xi- \bigg(1+ \f{1}{2} \al^2\bigg)  \eps \int \partial_{\xi}^6 \Phi \cdot \partial^2_{\xi} \bigg[\f{ (\Phi_{\xi \xi})^2}{(1+ \eps^3 (\Phi_{\xi})^2)^{\f{5}{2}}}\bigg] d\xi \\ \nonumber
	&+&   6 \al^2(\al+ 3) \eps^3 \int \partial_{\xi}^6 \Phi \cdot \partial^2_{\xi} \bigg[\f{ \Phi_{\xi} \Phi_{\xi \xi} \Phi_{\xi \xi \xi} }{(1+ \eps^3 (\phi_{\xi})^2)^3} \bigg] \partial^4_{\xi} \Phi d\xi+  3 \al^2(\al+ 3) \eps^3 \int  \partial_{\xi}^6 \Phi \cdot \partial^2_{\xi} \bigg[\f{ (\Phi_{\xi \xi})^3  }{(1+ \eps^3 (\phi_{\xi})^2)^3} \bigg]  d\xi\\ \nonumber
	&+& 4 \al^2(\al+ 3) \eps^3 \int \f{(\partial^6 \Phi) (\partial^5 \Phi) (\Phi_{\xi}) (\Phi_{ \xi \xi })}{(1+ \eps^3 (\Phi_{ \xi })^2 )^3} d\xi+ 4 \al^2(\al+ 3) \eps^3 \int (\partial_{\xi}^6 \Phi)    (\partial^4 \Phi) \partial_{\xi} \bigg[\f{ (\Phi_{\xi}) (\Phi_{ \xi \xi })}{(1+ \eps^3 (\Phi_{ \xi })^2 )^3}\bigg] d\xi
	\\ \nonumber
	&-&  \bigg(2 \al+ 5 \al^2- \f{1}{3} \al^3\bigg) \eps^3  \int \partial_{\xi}^6 \Phi \cdot \partial^2_{\xi} \bigg[ \f{  (\Phi_{\xi \xi})^3}{(1+ \eps^3 (\Phi_{\xi})^2)^4} \bigg] d\xi - 18 \al^2(\al+ 3) \eps^6 \int \partial_{\xi}^6 \Phi \cdot \partial^2_{\xi} \bigg[\f{  (\Phi_{\xi \xi})^3 (\Phi_{\xi})^2 }{(1+ \eps^3 (\phi_{\xi})^2)^4} \bigg]  d\xi\\ \nonumber
	&:=& J_1+ \cdots+ J_9.
	\end{eqnarray}
	As we argued in \eqref{es7}, we work with a simpler version of this energy estimate
		\begin{eqnarray}\label{en1}
\frac{1}{2} \partial_t \|\partial_x^4 \Phi\|_{L^2}^2+ C_0 \al^2(\al+ 3)  \int  \|\partial_\xi^6 \Phi\|_{L^2} d\xi
	&\leq& \bigg|J_1+ \cdots+ J_9\bigg|.
	\end{eqnarray}
	 As it is stated before, in each term we present the bound for the worse part of the integral, and that happens when in the integrand,  the two derivatives hit the highest degree in the numerator of the fraction. We denote the rest of the terms $G(\tau)$ which letter $G$ stands for good terms. One type of such (good) terms arises when derivatives hit the denominator. Any time a derivative is applied to the denominator, which is of the form $\frac{1}{\bigg(1+ \eps^3 (\Phi_{\xi})^2 \bigg)^{a}}$, it multiplies the  integrand in  $\frac{\eps^3 \Phi_{ \xi } \Phi_{ \xi \xi }}{\bigg(1+ \eps^3 (\Phi_{\xi})^2 \bigg)^{a+ 1}}$. These kind of terms are controlled in the following way,
	\begin{eqnarray*}
	\bigg\|\f{\eps^3 \Phi_{ \xi } \Phi_{ \xi \xi }}{\bigg(1+ \eps^3 (\Phi_{\xi})^2 \bigg)^{a+ 1}}\bigg\|_{L^{\infty}} \leq \eps^{\f{3}{2}} \bigg\|\f{\eps^{\f{3}{2}} \Phi_{\xi} }{\bigg(1+ \eps^3 (\Phi_{\xi})^2 \bigg)^{a+ 1}}\bigg\|_{L^{\infty}} \|\Phi_{ \xi \xi }\|_{L^{\infty}} \leq C \eps^{\f{3}{2}} \|\Phi_{ \xi \xi }\|_{L^{\infty}} \leq  C \eps^{\f{3}{2}} \||\nabla|^{\f{5}{2}}\Phi\|_{L^2} \leq C \eps^{\f{3}{2}} \|\Phi\|^{\f{3}{2}}_{L^2} \|\partial^4 \Phi\|^{\f{5}{8}}_{L^2}.
	\end{eqnarray*}
Although this might increase the power of  $ \|\Phi\|_{L^2}$ and  $\|\partial_x^4 \Phi\|_{L^2}$ in our final calculations, it also adds the power $\eps$ in front of every such terms, which fits our Gronwall's inequality \eqref{grr}.  For the rest of the proof, we ignore the good terms $G(\tau)$, and in each integral in \eqref{en11}, we present the proper bound   for the worse term. \\
	\textbf{Estimate for $J_1$:}
	 \begin{eqnarray*}
	 	|J_1|&=&  \bigg| \int \partial_{\xi}^6 \Phi \cdot  \partial_{\xi}^2 \bigg[\f{ \Phi_{\xi \xi}}{1+ \eps^3 (\Phi_{ \xi })^2 }\bigg] d\xi\bigg| \leq C \|\partial_{\xi}^4 \Phi\|_{L^2} \| \partial_{\xi}^6 \Phi\|_{L^2} \bigg\|\f{1}{1+ \eps^3 (\Phi_{\xi})^2}\bigg\|_{L^{\infty}}+ G(\tau)\\
	 	 &\leq& \f{C_0 \al^2(\al+ 3)}{100} \|\partial_{\xi}^6 \Phi\|^2_{L^2}+  C \bigg( \|\Phi\|^2_{L^2}+ \|\partial_{\xi}^4 \Phi\|^2_{L^2} \bigg)+ G(\tau)
	 \end{eqnarray*}
	\textbf{Estimate for $J_2$:}  
	\begin{eqnarray*}
		|J_2|&=&  \bigg|  \int \partial_{\xi}^6 \Phi \cdot \partial_{\xi}^2 \bigg[\frac{(\Phi_{ \xi })^2}{1+ \sqrt{1+ \eps^3 (\Phi_{\xi})^2}} \bigg]  d\xi\bigg| \leq   \|\partial_{\xi}^6 \Phi\|_{L^2} \bigg\| \f{1}{1+ \sqrt{1+ \eps^3 (\Phi_{\xi})^2}}\bigg\|_{L^{\infty}} \|\Phi_{\xi}\|_{L^{\infty}} \|\partial_{\xi}^3 \Phi\|_{L^2}+ G(\tau)\\
		& \leq& \|\partial_{\xi}^6 \Phi\|_{L^2}   \||\nabla|^{\f{3}{2}}\Phi\|_{L^2} \|\partial_{\xi}^3 \Phi\|_{L^2} \leq \|\partial_{\xi}^6 \Phi\|_{L^2} \bigg(\|\Phi\|^{\f{5}{8}}_{L^2} \|\partial_{\xi}^4 \Phi\|^{\f{3}{8}}_{L^2} \bigg)  \bigg(\|\Phi\|^{\f{1}{4}}_{L^2} \|\partial_{\xi}^4 \Phi\|^{\f{3}{4}}_{L^2} \bigg)+ G(\tau)\\
		& \leq& \f{C_0 \al^2(\al+ 3)}{100} \|\partial_x^6 \Phi\|^2_{L^2} + C \bigg( \|\partial_{\xi}^4 \Phi\|^4_{L^2}+ \| \Phi\|^4_{L^2}\bigg)+ G(\tau).
	\end{eqnarray*}

	\textbf{Estimate for $J_3$:} 
	\begin{eqnarray*}
		|J_3|&=&    \bigg(1+ \f{1}{2} \al^2\bigg)  \eps \bigg|\int \partial_{\xi}^6 \Phi \cdot \partial^2_{\xi} \bigg[\f{ (\Phi_{\xi \xi})^2}{(1+ \eps^3 (\Phi_{\xi})^2)^{\f{5}{2}}}\bigg] d\xi\bigg| \leq C \eps  \|\partial_{\xi}^6 \Phi\|_{L^2}  \|\Phi_{\xi \xi}\|_{L^{\infty}} \|\partial_{\xi}^4 \Phi\|_{L^2}+ G(\tau)\\
		&\leq& C \eps  \|\partial_{\xi}^6 \Phi\|_{L^2}  \||\nabla|^{\f{5}{2}} \Phi\|_{L^2} \|\partial_{\xi}^4 \Phi\|_{L^2}+ G(\tau)  \leq C \eps \|\partial_{\xi}^6 \Phi\|_{L^2}  \bigg(\|\Phi\|^{\f{3}{8}}_{L^2} \|\partial_{\xi}^4 \Phi\|^{\f{5}{8}}_{L^2} \bigg)  \|\partial_{\xi}^4 \Phi\|_{L^2}+ G(\tau)\\
		 &\leq& \f{C_0 \al^2(\al+ 3)}{100}  \|\partial_x^6 \Phi\|^2_{L^2} + C \eps^2 \bigg( \| \Phi\|^4_{L^2}+  \|\partial^4 \Phi\|^4_{L^2}\bigg)+ G(\tau).
	\end{eqnarray*}

	\textbf{Estimate for $J_4$:} 
	\begin{eqnarray*}
		|J_4|&=&    6 \al^2(\al+ 3) \eps^3 \bigg| \int \partial_{\xi}^6 \Phi \cdot \partial^2_{\xi} \bigg[\f{ \Phi_{\xi} \Phi_{\xi \xi} \Phi_{\xi \xi \xi} }{(1+ \eps^3 (\phi_{\xi})^2)^3} \bigg]  d\xi\bigg| \leq \eps^3 \|\partial_{\xi}^6 \Phi\|_{L^2}     \|\Phi_{\xi}\|_{L^{\infty}} \|\Phi_{\xi \xi}\|_{L^{\infty}}  \|\partial_{\xi}^5 \Phi\|_{L^2}+ G(\tau)\\
		&\leq& \eps^3 \|\partial_{\xi}^6 \Phi\|_{L^2} \||\nabla|^{\f{3}{2}} \Phi\|_{L^2}  \||\nabla|^{\f{5}{2}} \Phi\|_{L^2}  \|\partial_{\xi}^5 \Phi\|_{L^2}\\
		& \leq&  \eps^3 \|\partial_{\xi}^6 \Phi\|_{L^2}  \bigg(\|\Phi\|^{\f{5}{8}}_{L^2} \|\partial_{\xi}^4 \Phi\|^{\f{3}{8}}_{L^2} \bigg)  \bigg(\|\Phi\|^{\f{3}{8}}_{L^2} \|\partial_{\xi}^4 \Phi\|^{\f{5}{8}}_{L^2} \bigg)   \bigg( \|\partial_{\xi}^4 \Phi\|^{\f{1}{2}}_{L^2} \|\partial{\xi}^6 \Phi\|^{\f{1}{2}}_{L^2} \bigg)+ G(\tau) \\
		&\leq&  C \eps^{3} \|\partial_{\xi}^6 \Phi\|^{\f{3}{2}}_{L^2}  \bigg( \| \Phi\|_{L^2} \|\partial_{\xi}^4 \Phi\|^{\f{3}{2}}_{L^2}  \bigg)+ G(\tau)  \leq  \f{C_0 \al^2(\al+ 3)}{100}  \|\partial_{\xi}^6 \Phi\|^2_{L^2}+ \eps^{12}  \bigg( \| \Phi\|^4_{L^2} \|\partial_{\xi}^4 \Phi\|^6_{L^2}  \bigg)+ G(\tau)  \\
		& \leq& \f{C_0 \al^2(\al+ 3)}{100}  \|\partial_{\xi}^6 \Phi\|^2_{L^2}+ \eps^{12}   \bigg(\|\Phi\|^{10}_{L^2}+ \|\partial_{\xi}^4 \Phi\|^{10}_{L^2} \bigg)+ G(\tau).
	\end{eqnarray*}
	\textbf{Estimate for $J_5$:} 
	\begin{eqnarray*}
	|J_5|&=& 3 \al^2(\al+ 3) \eps^3 \bigg|\int  \partial_{\xi}^6 \Phi \cdot \partial^2_{\xi} \bigg[\f{ (\Phi_{\xi \xi})^3  }{(1+ \eps^3 (\phi_{\xi})^2)^3} \bigg]  d\xi \bigg| \\
	&\leq&  \eps^3 \bigg\| \frac{1}{(1+ \eps^3 (\Phi_{\xi})^2)^3}\bigg\|_{L^{\infty}}  \|\partial_{\xi}^6 \Phi\|_{L^2} \|\partial_{\xi}^4 \Phi\|_{L^2}  \| \Phi_{\xi \xi}\|^2_{L^{\infty}}+ G(\tau) \\
	& \leq& \f{C_0 \al^2(\al+ 3)}{100}  \|\partial_{\xi}^6 \Phi\|^2_{L^2}+ C \eps^{6}   \bigg(\|\Phi\|^6_{L^2}+ \|\partial_{\xi}^4 \Phi\|^6_{L^2} \bigg)+ G(\tau)\\
		& \leq& \f{C_0 \al^2(\al+ 3)}{100}  \|\partial_{\xi}^6 \Phi\|^2_{L^2}+ C \bigg(\|\Phi\|^2_{L^2}+ \|\partial_{\xi}^4 \Phi\|^2_{L^2} \bigg)+ C \eps^{12}   \bigg(\|\Phi\|^{10}_{L^2}+ \|\partial_{\xi}^4 \Phi\|^{10}_{L^2} \bigg)+ G(\tau).
	\end{eqnarray*}
	Now we have use relation 
	
		\textbf{Estimate for $J_6$:}  
	\begin{eqnarray*}
		|J_6|&=&      4 \al^2(\al+ 3) \eps^3 \bigg|   \int \f{(\partial_{\xi}^6 \Phi) (\partial_{\xi}^5 \Phi) (\Phi_{\xi}) (\Phi_{ \xi \xi })}{(1+ \eps^3 (\Phi_{ \xi })^2 )^3} d\xi\bigg|\\
		& \leq& C \eps^3 \bigg\| \frac{1}{(1+ \eps^3 (\Phi_{\xi})^2)^3}\bigg\|_{L^{\infty}}  \|\partial_{\xi}^6 \Phi\|_{L^2}  \|\partial_{\xi}^5 \Phi\|_{L^2}    \|\Phi_{\xi \xi}\|_{L^{\infty}} \| \Phi_{ \xi }\|_{L^{\infty}}+ G(\tau) \\
		&\leq& C \eps^3    \|\partial_{\xi}^6 \Phi\|_{L^2}   \|\partial_{\xi}^5 \Phi\|_{L^2}   \||\nabla|^{\f{5}{2}} \Phi\|_{L^2}    \||\nabla|^{\f{3}{2}} \Phi\|_{L^2}\\
		& \leq& C \eps^3    \|\partial_{\xi}^6 \Phi\|_{L^2} \bigg(\|\partial_{\xi}^4 \Phi\|^{\f{1}{2}}_{L^2} \|\partial_{\xi}^6 \Phi\|^{\f{1}{2}}_{L^2}  \bigg) \bigg(\|\Phi\|^{\f{3}{8}}_{L^2} \|\partial_{\xi}^4 \Phi\|^{\f{5}{8}}_{L^2} \bigg) \bigg(\|\Phi\|^{\f{5}{8}}_{L^2} \|\partial_{\xi}^4 \Phi\|^{\f{3}{8}}_{L^2} \bigg)+ G(\tau)\\ 
		 & \leq& C \eps^3    \|\partial_{\xi}^6 \Phi\|_{L^2}^{\f{3}{2}}  \bigg(\|\Phi\|_{L^2} \|\partial_{\xi}^4 \Phi\|^{\f{3}{2}}_{L^2} \bigg)+ G(\tau)  \leq \f{C_0 \al^2(\al+ 3)}{100}   \|\partial_{\xi}^6 \Phi\|^2_{L^2}+ \eps^{12}   \bigg(\|\Phi\|^{10}_{L^2}+ \|\partial_{\xi}^4 \Phi\|^{10}_{L^2} \bigg)+ G(\tau).
	\end{eqnarray*}
	\textbf{Estimate for $J_7$:}  
		\begin{eqnarray*}
		|J_7|&=&   4 \eps^3 \al^2(\al+ 3)  \bigg|  \int \partial_{\xi}^6 \Phi \cdot \partial_{\xi}^4 \Phi \partial_{\xi} \cdot \bigg[\f{  (\Phi_{\xi \xi}) (\Phi_{\xi}) }{(1+ \eps^3 (\phi_{\xi})^2)^3} \bigg]  d\xi\bigg|\\
		& \leq& C \eps^{\f{3}{2}}  \|\partial_{\xi}^6 \Phi\|_{L^2}  \|\partial_{\xi}^4 \Phi\|_{L^2}  \| \Phi_{\xi \xi \xi}\|_{L^{\infty}} \|\f{ \eps^{\f{3}{2}}  \Phi_{\xi} }{(1+ \eps^3 (\Phi_{\xi}))^2)^3}\|_{L^{\infty}} + G(\tau)  \\
		&\leq&  C \eps^{6}  \|\partial_{\xi}^6 \Phi\|_{L^2}    \|\partial_{\xi}^4 \Phi\|_{L^2}  \bigg(\|\Phi\|^{\f{1}{8}}_{L^2} \|\partial_{\xi}^4 \Phi\|^{\f{7}{8}}_{L^2} \bigg)+ G(\tau)\\
		&\leq& \f{C_0 \al^2(\al+ 3)}{100}   \|\partial_{\xi}^6 \Phi\|^2_{L^2}+ C \eps^{3}   \bigg(\|\Phi\|^4_{L^2}+ \|\partial_{\xi}^4 \Phi\|^4_{L^2} \bigg)+ G(\tau).
	\end{eqnarray*}
	\textbf{Estimate for $J_8$:}  Similar to  $J_5$.\\
	\textbf{Estimate for $J_9$:} 
	\begin{eqnarray*}
		|J_9|&=&   18 \al^2(\al+ 3) \eps^6 \bigg|  \int \partial_{\xi}^6 \Phi \cdot \partial^2_{\xi} \bigg[\f{  (\Phi_{\xi \xi})^3 (\Phi_{\xi})^2 }{(1+ \eps^3 (\phi_{\xi})^2)^3} \bigg]  d\xi\bigg| \leq C \eps^6  \|\partial_{\xi}^6 \Phi\|_{L^2}  \|\partial_{\xi}^4 \Phi\|_{L^2}  \| \Phi_{\xi \xi}\|^2_{L^{\infty}} \| \Phi_{\xi}\|^2_{L^{\infty}} + G(\tau)  \\
		&\leq& C \eps^6  \|\partial_{\xi}^6 \Phi\|_{L^2}  \|\partial_{\xi}^4 \Phi\|_{L^2}  \| |\nabla|^{\f{5}{2}}\Phi\|^2_{L^{2}}  \| |\nabla|^{\f{3}{2}}\Phi\|^2_{L^{2}} + G(\tau) \\
		&\leq&  C \eps^{6}  \|\partial_{\xi}^6 \Phi\|_{L^2}    \|\partial_{\xi}^4 \Phi\|_{L^2}  \bigg(\|\Phi\|^{\f{3}{8}}_{L^2} \|\partial_{\xi}^4 \Phi\|^{\f{5}{8}}_{L^2} \bigg)^2 \bigg(\|\Phi\|^{\f{5}{8}}_{L^2} \|\partial_{\xi}^4 \Phi\|^{\f{3}{8}}_{L^2} \bigg)^2+ G(\tau)\\
 &\leq& \f{C_0 \al^2(\al+ 3)}{100}   \|\partial_{\xi}^6 \Phi\|^2_{L^2}+ \eps^{12}   \bigg(\|\Phi\|^{10}_{L^2}+ \|\partial_{\xi}^4 \Phi\|^{10}_{L^2} \bigg)+ G(\tau).
\end{eqnarray*}
Therefore, we can summarize the energy estimates \eqref{en11} in the following form
\begin{eqnarray}\label{gr1}
\f{1}{2} \partial_t \| \partial_{\xi}^4 \Phi\|^2_{L^2} \leq C \bigg( \|\Phi\|^2_{L^2}+ \|\partial_{\xi}^4 \Phi\|^2_{L^2}  \bigg)+ \bigg( \|\Phi\|^4_{L^2}+ \|\partial_{\xi}^4 \Phi\|^4_{L^2}  \bigg)+  \eps^{12} \bigg( \|\Phi\|^{10}_{L^2}+ \|\partial_{\xi}^4 \Phi\|^{10}_{L^2}  \bigg)+ G(\tau).
\end{eqnarray}
At this point we combine both energy estimates \eqref{gr0} and \eqref{gr1}
\begin{eqnarray*}
\frac{1}{2} \partial_t  \bigg(\|\Phi\|_{L^2}^2+ \|\partial_{\xi}^4 \Phi\|_{L^2}^2\bigg) \leq  C \bigg( \|\Phi\|^2_{L^2}+ \|\partial_{\xi}^4 \Phi\|^2_{L^2}  \bigg)+ C \bigg( \|\Phi\|^4_{L^2}+ \|\partial_{\xi}^4 \Phi\|^4_{L^2}  \bigg)+ C \eps^{12} \bigg( \|\Phi\|^{10}_{L^2}+ \|\partial_{\xi}^4 \Phi\|^{10}_{L^2}  \bigg)+ G(\tau).
\end{eqnarray*}
	Note that $G(\tau)$ is also bounded by a combination of the terms in the form of $\eps^a \bigg( \|\Phi\|_{L^2}^b+ \|\partial_{\xi}^4 \Phi\|_{L^2}^b\bigg)$. We define $E(t)= \|\Phi\|_{L^2}^2+ \|\partial_{\xi}^4 \Phi\|_{L^2}^2$. Then this inequality is clearly in the form of the Gronwall's inequality in Lemma \ref{grr}, and it finishes the proof.
\end{proof}

\subsection{Asymptotics} 
In this section we show that the solutions of the scaled equations \eqref{UU} and \eqref{Phi0} stay close up to a time $\tau_*$.  In the previous section established the existence of the solution of the equation \eqref{Phi0} in $H^4$ on a time interval $[0, \tau_*]$, under some restrictions.  We also recall an important  result of the global boundedness of the function $U(\xi, \tau)$ in any Sobolev spaces. This result is proved by Tadmor \cite{TA}. 
\begin{lem}\label{tad}
	The (KS) equation \eqref{UU} with the initial value $U_0 \in H^4$ admits a global smooth solution 
	 \begin{eqnarray}
	U(\xi, \tau) \in H^4.
	\end{eqnarray}
\end{lem}
\begin{lem}\label{PhiU} Fix $\tau_*>0$ and $\Gamma_*>0$ and take $E_*$ and $\ep_*$ as in Lemma \ref{lem3}. Assume that $\| \Phi(0)\|_{H^4} \le E_*$ and $0<\ep <\ep_*$.
	Let $U(\xi, \tau)$ and $\Phi(\xi, \tau)$ be the solutions of the equations \eqref{UU} and \eqref{Phi0}  respectively, where we assume $\|U(0)- \Phi(0)\|_{L^2} \leq \eps$.
	Then
	\begin{eqnarray} \sup_{t \in [0,\tau_*]}
	\|\Phi(t)- U(t)\|_{L^2} \leq \Gamma_{**} \eps.
	\end{eqnarray}
	The constant $\Gamma_{**}>0$ does not depend on $\ep$.
\end{lem}  

\begin{rem}\label{remproof} Lemma \ref{PhiU} 
leads directly to Theorem \ref{thm2}. Here is the calculation.
Recalling that $y(x,t) = \eps \Phi(\sqrt{\eps}x, \eps^2 t)-t$. 
we have
\begin{equation}
\begin{split}
\sup_{0< t \le \tau_0/\eps^2} \| y(\cdot,t) + t - \eps U(\sqrt{\eps} \cdot,\eps^2 t)\|_{L^2} &= 
\sup_{0< t \le \tau_0/\eps^2} \eps \| \Phi(\sqrt{\eps} \cdot,\eps^2 t) - U(\sqrt{\eps} \cdot,\eps^2 t)\|_{L^2} \\
&=\sup_{0< \tau \le \tau_0} \eps^{3/4}\| \Phi( \cdot,\tau) - U( \cdot,\tau)\|_{L^2}\\
&\le C \eps^{7/4}.
\end{split}
\end{equation}
This is the concluding estimate in Theorem \ref{thm2}.
In the above we used the change of variables relation $\| f(\alpha \cdot)\|_{L^2} = \alpha^{-1/2} \| f(\cdot)\|_{L^2}$
for $\alpha > 0$. 
\end{rem}

\begin{proof}
	From the equations \eqref{UU} and \eqref{Phi0} we construct  the equation for the quantity $v= \Phi- U$. Indeed, since for $\al= \eps+ 1$, $\al^2 (\al+ 3)= 4+ \eps (\eps+ 3)^2$, we have 
	\begin{eqnarray*}
		\partial_{\tau} (\Phi- U)&+& 4 \bigg[\f{\Phi_{\xi \xi \xi \xi}}{\bigg(1+ \eps^3 (\Phi_{ \xi })^2 \bigg)^2}-  U_{\xi \xi \xi \xi}\bigg]+ \f{1}{\eps^3} \bigg[ \sqrt{1+ \eps^3 (\Phi_{ \xi })^2}- 1- \f{\eps^3}{2} (U_{\xi})^2\bigg]+ (\al- 1) \bigg[\f{\Phi_{\xi \xi}}{1+ \eps^3 (\Phi_{ \xi })^2}-  U_{\xi \xi}\bigg]\\
		&=&   10 \al^2(\al+ 3) \f{\eps^3 \Phi_{\xi} \Phi_{\xi \xi} \Phi_{\xi \xi \xi}}{(1+ \eps^3 (\Phi_{\xi})^2)^3}+  3 \al^2(\al+ 3) \f{\eps^3 (\Phi_{\xi \xi})^3}{(1+ \eps^3 (\Phi_{\xi})^2)^3}-   18 \al^2(\al+ 3) \f{\eps^6 (\Phi_{\xi \xi})^3 (\Phi_{\xi})^2}{(1+ \eps^3 (\Phi_{\xi})^2)^4}\\
		&-& \bigg(1+ \f{1}{2} \al^2\bigg) \f{\eps (\Phi_{\xi \xi})^2}{(1+ \eps^3 (\Phi_{\xi})^2)^{\f{5}{2}}}- \bigg(2 \al+ 5 \al^2- \f{1}{3} \al^3\bigg)  \f{\eps^3 (\Phi_{\xi \xi})^3}{(1+ \eps^3 (\Phi_{\xi})^2)^4}- \eps (\eps+ 3)^2 \f{\Phi_{\xi \xi \xi \xi}}{\bigg(1+ \eps^3 (\Phi_{ \xi })^2 \bigg)^2}.
	\end{eqnarray*}
	Then we can simplify it in the following form
	\begin{eqnarray*}
		\partial_{\tau} v&+& 4 \bigg[\f{v_{\xi \xi \xi \xi}}{\bigg(1+ \eps^3 (\Phi_{ \xi })^2 \bigg)^2}\bigg]+ \bigg[\f{v_{\xi} \cdot (\Phi_{\xi}+ U_{\xi})}{\bigg( \sqrt{1+ \eps^3 (\Phi_{ \xi })^2}\bigg)+ \bigg( 1+ \f{\eps^3}{2} (U_{\xi})^2 \bigg)}\bigg]+ \bigg[\f{v_{\xi \xi}}{1+ \eps^3 (\Phi_{ \xi })^2}\bigg]\\
		&=& \al^2(\al+ 3) \bigg[ \f{2 \eps^3 (\Phi_{\xi})^2+ \eps^6 (\Phi_{ \xi })^4}{(1+ \eps^3 (\Phi_{ \xi })^2 )^2}\bigg] U_{\xi \xi \xi \xi }+  \f{1}{4} \cdot \f{\eps^3 (U_{\xi})^4}{\bigg( \sqrt{1+ \eps^3 (\Phi_{ \xi })^2}\bigg) +\bigg( 1+ \f{\eps^3}{2} (U_{\xi})^2 \bigg)}+ (\al- 1) \f{\eps^3 (\Phi_{ \xi })^2 U_{\xi \xi}}{1+ \eps^3 (\Phi_{ \xi })^2}\\
		&+& 10 \al^2(\al+ 3) \f{\eps^3 \Phi_{\xi} \Phi_{\xi \xi} \Phi_{\xi \xi \xi}}{(1+ \eps^3 (\Phi_{\xi})^2)^3}+  3 \al^2(\al+ 3) \f{\eps^3 (\Phi_{\xi \xi})^3}{(1+ \eps^3 (\Phi_{\xi})^2)^3}-   18 \al^2(\al+ 3) \f{\eps^6 (\Phi_{\xi \xi})^3 (\Phi_{\xi})^2}{(1+ \eps^3 (\Phi_{\xi})^2)^4}\\
		&-& \bigg(1+ \f{1}{2} \al^2\bigg) \f{\eps (\Phi_{\xi \xi})^2}{(1+ \eps^3 (\Phi_{\xi})^2)^{\f{5}{2}}}- \bigg(2 \al+ 5 \al^2- \f{1}{3} \al^3\bigg)  \f{\eps^3 (\Phi_{\xi \xi})^3}{(1+ \eps^3 (\Phi_{\xi})^2)^4} - \eps (\eps+ 3)^2 \f{\Phi_{\xi \xi \xi \xi}}{\bigg(1+ \eps^3 (\Phi_{ \xi })^2 \bigg)^2},
	\end{eqnarray*}
	with the initial condition $v(\xi, 0)= 0$. The presence of at least one $\eps$ in the right hand side of this relation, as well as the $H^4$ bounds for both $U(\xi, \tau)$ and $\Phi(\xi, \tau)$  makes the right hand side very convenient. For future calculations we give the right had side a name, say $\eps F(\xi, \tau)$. It is not very difficult to see that for any time $\tau \in [0, \tau_*]$ we have
	\begin{eqnarray}
	\|F\|_{L^2} \leq C.
	\end{eqnarray} 
	
	 Now we find the inner product of the above equation with $v$,
	\begin{eqnarray}\label{en8}
	\f{1}{2} \partial_{\tau} \|v\|^2_{L^2}&+&  \al^2(\al+ 3) \int v \cdot \bigg[\f{v_{\xi \xi \xi \xi}}{\bigg(1+ \eps^3 (\Phi_{ \xi })^2 \bigg)^2}\bigg] d\xi
	=-  \int v \cdot \bigg[\f{v_\xi \cdot (\Phi_{\xi}+ U_{\xi})}{\bigg(\sqrt{1+ \eps^3 (\Phi_{ \xi })^2}\bigg)+ \bigg( 1+ \f{\eps^3}{2} (U_{\xi})^2 \bigg)}\bigg] d\xi\\ \nonumber
	&-& \int v \cdot \bigg[\f{v_{\xi \xi}}{1+ \eps^3 (\Phi_{ \xi })^2}\bigg] d\xi+ \eps \int v \cdot F(\xi, t) d\xi.
	\end{eqnarray}
	Then, 
	\begin{eqnarray*}
		&& \int v \cdot \bigg[\f{v_{\xi \xi \xi \xi}}{\bigg(1+ \eps^3 (\Phi_{ \xi })^2 \bigg)^2}\bigg] d\xi= \int \f{(v_{\xi \xi})^2}{(1+ \eps^3 (\Phi_{ \xi })^2)^2} d\xi\\
		&&- 4 \eps^3 \int  v_{\xi \xi} \bigg[\f{ 2 v_{\xi} \Phi_{ \xi} \Phi_{ \xi \xi }+ v_{\xi}  \Phi_{ \xi} \Phi_{ \xi \xi }+ v  (\Phi_{ \xi \xi })^2+ v  \Phi_{ \xi} \Phi_{ \xi \xi \xi } }{(1+ \eps^3 (\Phi_{ \xi })^2)^3}\bigg] d\xi+  24 \eps^6 \int v_{\xi \xi} \bigg[\f{v (\Phi_{ \xi})^2 (\Phi_{ \xi \xi })^2 }{(1+ \eps^3 (\Phi_{ \xi })^2)^4}\bigg] d\xi.
	\end{eqnarray*}
Considering the relation \eqref{b0}, we can present a lower bound for the major part of this equality, i.e,
\begin{eqnarray}
\int \f{(v_{\xi \xi})^2}{(1+ \eps^3 (\Phi_{ \xi })^2)^2} d\xi \geq \|v_{\xi \xi}\|^2_{L^2}.
\end{eqnarray}
	Therefore the energy estimate \eqref{en8} turns into,
	\begin{eqnarray*}
		\f{1}{2} \partial_{\tau} \|v\|^2_{L^2}&+&  \al^2 (\al+ 3) \|v_{\xi \xi}\|^2_{L^2}
		= - \int v \cdot \bigg[\f{v_\xi \cdot (\Phi_{\xi}+ U_{\xi})}{\bigg( \sqrt{1+ \eps^3 (\Phi_{ \xi })^2}\bigg)+ \bigg( 1+ \f{\eps^3}{2} (U_{\xi})^2 \bigg)}\bigg] d\xi\\ \nonumber
		&-& \int v \cdot \bigg[\f{v_{\xi \xi}}{1+ \eps^3 (\Phi_{ \xi })^2}\bigg] d\xi+ 4 \al^3 (\al+ 3) \eps^3\int  v_{\xi \xi} \bigg[\f{ 2 v_{\xi} \Phi_{ \xi} \Phi_{ \xi \xi }+ v_{\xi}  \Phi_{ \xi} \Phi_{ \xi \xi }+ v  (\Phi_{ \xi \xi })^2+ v  \Phi_{ \xi} \Phi_{ \xi \xi \xi } }{(1+ \eps^3 (\Phi_{ \xi })^2)^3}\bigg] d\xi\\
		&+& 24 \al^3 (\al+ 3) \eps^6 \int  v_{\xi \xi} \bigg[ \f{v (\Phi_{ \xi})^2 (\Phi_{ \xi \xi })^2 }{(1+ \eps^3 (\Phi_{ \xi })^2)^4} \bigg] d\xi+  \eps \int v \cdot F(\xi, t) d\xi= K_1+ K_2+ K_3+ K_4+ K_5.
	\end{eqnarray*}
	Now we find proper bounds for the right hand side of this relation. \\
	{\bf Estimate for $K_1$:} Considering the relation $v v_{\xi}= \f{1}{2} \partial_{\xi} (v^2)$ we have
	\begin{eqnarray*}
		|K_1|&\leq& \bigg|\int v\cdot \bigg[\f{v_{\xi}  \cdot (\Phi_{\xi}+ U_{\xi})}{\bigg(1+ \sqrt{ \eps^3 (\Phi_{ \xi })^2}\bigg)+ \bigg( 1+ \f{\eps^3}{2} (U_{\xi})^2 \bigg)}\bigg] d\xi \bigg|= \bigg|\int v^2 \cdot \partial_{\xi} \bigg[\f{ (\Phi_{\xi}+ U_{\xi})}{\bigg(1+ \sqrt{ \eps^3 (\Phi_{ \xi })^2}\bigg)+ \bigg( 1+ \f{\eps^3}{2} (U_{\xi})^2 \bigg)}\bigg] d\xi \bigg|\\
		 &&\leq \|v\|_{L^2}^2 \bigg\| \f{\Phi_{\xi}+ U_{\xi}}{\bigg( \sqrt{1+ \eps^3 (\Phi_{ \xi })^2}\bigg)+ \bigg( 1+ \f{\eps^3}{2} (U_{\xi})^2 \bigg)}\bigg\|_{L^{\infty}} < C \|v\|_{L^2}^2.
	\end{eqnarray*}
{\bf Estimate for $K_2$:}
	\begin{eqnarray*}
		|K_2| \leq \bigg|\int v \cdot \bigg[\f{v_{\xi \xi}}{1+ \eps^3 (\Phi_{ \xi })^2}\bigg] d\xi \bigg| \leq C \|v\|_{L^2}^2+ \f{1}{100} \|v_{\xi \xi}\|^2_{L^2}.
	\end{eqnarray*}
{\bf Estimate for $K_3$:}
\begin{eqnarray*}
|K_3|  & \leq& C  \eps^3 \bigg| \int  v_{\xi \xi} \bigg[\f{ 2 v_{\xi} \Phi_{ \xi} \Phi_{ \xi \xi }+ v_{\xi}  \Phi_{ \xi} \Phi_{ \xi \xi }+ v  (\Phi_{ \xi \xi })^2+ v  \Phi_{ \xi} \Phi_{ \xi \xi \xi } }{(1+ \eps^3 (\Phi_{ \xi })^2)^3}\bigg] d\xi\bigg| \leq C \eps^3 \|v_{\xi \xi}\|_{L^2} (\|v\|_{L^2}+ \|v_{\xi}\|_{L^2})\\
& \leq&  \f{1}{100} \|v_{\xi \xi}\|^2_{L^2}+ C \eps^6 \|v\|^2_{L^2}+ C \eps^6.
\end{eqnarray*}
	Note that all the terms  $\partial_{\xi}^s \Phi$, $1 \leq s \leq 3$, are bounded (since $\Phi \in H^4$). \\
{\bf Estimate for $K_4$:}
\begin{eqnarray*}
	|K_4|   \leq C  \eps^6 \bigg| \int  v_{\xi \xi} \bigg[ \f{v (\Phi_{ \xi})^2 (\Phi_{ \xi \xi })^2 }{(1+ \eps^3 (\Phi_{ \xi })^2)^4} \bigg] d\xi\bigg| \leq C \eps^6 \|v_{\xi \xi}\|_{L^2} \|v\|_{L^2} \leq  \f{1}{100} \|v_{\xi \xi}\|^2_{L^2}+ C \eps^{12} \|v\|^2_{L^2}.
\end{eqnarray*}		
{\bf Estimate for $K_5$:}
\begin{eqnarray*}
	|K_5|  \leq C  \eps \bigg| \int v \cdot F(\xi, \tau) d\xi\bigg| \leq C \eps \| F(\cdot, \tau)\|_{L^2} \|v\|_{L^2} \leq  C  \|v\|^2_{L^2}+ C \eps^2.
\end{eqnarray*}		
	Overall, the energy estimate \eqref{en8} turns into 
	\begin{eqnarray*}
		\partial_{\tau} \|v\|^2_{L^2} + C \|v_{\xi \xi}\|_{L^2}^2 \leq C_1 \|v\|_{L^2}^2+ C_2 \eps^2,
	\end{eqnarray*}
	or 
	\begin{eqnarray*}
		\partial_{\tau} \|v\|^2_{L^2}  \leq C_1 \|v\|_{L^2}^2+ C_2 \eps^2.
	\end{eqnarray*}
	Then, , we take integral from both sides,
	\begin{eqnarray*}
		\|v\|^2_{L^2}  \leq e^{C_1 \tau} \|v(0)\|^2_{L^2}+  C_2 \eps^2 \int_0^{\tau} e^{C_0 (\tau- s)} ds=  e^{C_1 \tau} \|v(0)\|^2_{L^2}+ \f{C_2 \eps^2 }{C_1} \bigg[e^{C_1 \tau}- 1\bigg].
	\end{eqnarray*}
Finally, we restrict ourself to $\tau< \tau_0$, $\tau_0= O(1)$, as well as $\|v(0)\|_{L^2} \leq \eps$, and complete the proof. 
\end{proof}

\bibliographystyle{plain} 
\bibliography{ambroseHadadifardWright}{}
%
%
%
%
%
%
%
%
%

\end{document}